\documentclass[11pt, a4paper, UKenglish]{article}
\usepackage[UKenglish]{babel}
\usepackage[utf8]{inputenc}
\usepackage{amssymb}
\usepackage{amsmath}
\usepackage{amsthm}
\usepackage{tikz}
\usepackage{tikz-cd}

\newcommand{\nospacepunct}[1]{\makebox[0pt][l]{\,,}}
\usepackage{quiver}
\usepackage{hyperref}
\newtheorem{thm}{Theorem}[section]
\newtheorem{prop}[thm]{Proposition}
\newtheorem{lem}[thm]{Lemma}
\newtheorem{cor}[thm]{Corollary}
\theoremstyle{definition}

\theoremstyle{remark}

\theoremstyle{construction}

\theoremstyle{definition}

\usepackage{atbegshi,picture}
\usepackage{graphicx}
\begin{document}
	\title{
		{Motivic Toda brackets II}\\
	}
	\author{Xiaowen Dong\footnote{\href{mailto:xidong@uni-mainz.de}{xidong@uni-mainz.de}}}
	\maketitle
	\begin{abstract}
		We construct more non-trivial examples for Toda brackets in unstable motivic homotopy theory via the first and second motivic Hopf maps $\eta$ and $\nu$.
	\end{abstract}

		\tableofcontents
	
	\section{Introduction}
	\
	
This paper is a continuation of the author's doctoral thesis \cite{dong2024motivictodabrackets}. We  construct in the present article more examples of unstable motivic Toda brackets via the motivic Hopf maps $\eta$ and $\nu$.\\

First we fix some conventions. Throughout the paper we work entirely over the base $\mathrm{Spec}\mathbb{Z}$. Let $\mathcal{S}\mathrm{m}_{\mathbb{Z}}$ denote the category of smooth schemes of finite type over $\mathbb{Z}$. This category is equipped with the Nisnevich topology in the sense of \cite{Morel1999A1homotopyTO}. The category of motivic spaces $\mathrm{sPre}(\mathbb{Z})$ is the category of simplicial presheaves on $\mathcal{S}\mathrm{m}_{\mathbb{Z}}$. The category of pointed motivic spaces is the category of pointed simplicial presheaves on $\mathcal{S}\mathrm{m}_{\mathbb{Z}}$. It is denoted by $\mathrm{sPre}(\mathbb{Z})_{\ast}$. For the purposes of this paper we will use the $\mathbb{A}^1$-local injective model structure on  $\mathrm{sPre}(\mathbb{Z})_{\ast}$ which is developed in \cite{Morel1999A1homotopyTO}.\\

Since we want to construct examples uisng motivic spheres, we also recall here the definition of motivic spheres. Let $s, w\geq 0$ be integers. We define $S^{s+(w)}$ to be the pointed simplicial presheaf $S^{s}\wedge \mathbb{G}_{m}^w $ where $S^{s}$ is the smash product $\underbrace{S^1\wedge ... \wedge S^1}_{s \  times}$ of the simplicial circle $S^1= \Delta^1 /\partial\Delta^1$. The Tate circle $\mathbb{G}_{m} $ is $\mathrm{Spec}\  \mathbb{Z}[t]_{(t)}$. It is based at 1. We call $s$ the degree and $w$ the weight of $S^{s+(w)}$. Suspension from the right with $\mathbb{G}_{m} $ increases the weight $(w)$ by 1. Suspension from the left by the simplicial circle $S^1$ increases the degree $s$ by 1. Let $\mathcal{E}$ be an arbitrary pointed motivic space in $ \mathrm{sPre}(\mathbb{Z})_{\ast}$. Then we set $\pi_{s+(w)}\mathcal{E}$ to be the group $\mathcal{H}_{\bullet}(\mathbb{Z})(S^{s+(w)}, \mathcal{E}) $ for $s>0$ and $w\geq 0$. In particular, maps from spheres to spheres are always indexed by the bidegree of the target.\\

More precisely we would like to prove some unstable null-Hopf relations. Using these unstable null-Hopf relations we can then construct the corresponding motivic Toda brackets. Stable null-Hopf relations were already studied by Dugger and Isaksen in \cite{motihopf}. In order to state their results we recall the definition of the motivic Hopf elements $\eta$ and $\nu$. For two pointed motivic spaces $\mathcal{X}$ and $\mathcal{Y}$ the reduced join $\mathcal{X}\ast \mathcal{Y}$ is defined to be the quotient of $\Delta^1\times \mathcal{X} \times \mathcal{Y}  $ by the relations $(0,x,y)=(0,x,y'), (1,x,y)=(1,x',y)$ and $(t,x_0,y_0)=(s,x_0, y_0)$ where $x_0$ (resp.$y_0$) is the base point of $\mathcal{X}$ (resp.$\mathcal{Y}$). It is also the homotopy limit of the diagram of pointed spaces 
\[\begin{tikzcd}
	{\mathcal{X}\times \mathcal{Y}} & {\mathcal{X}} \\
	{\mathcal{Y}} & {}
	\arrow[from=1-1, to=1-2]
	\arrow[from=1-1, to=2-1]
\end{tikzcd}\]
Let $C(\mathcal{X})$ denote the cone $\Delta^1\wedge \mathcal{X}$ where $\Delta^1$ is based at 1. Let $C'(\mathcal{Y})$ denote $\Delta^1\wedge \mathcal{X}$ where $\Delta^1$ is based here at 0. There are canonical inclusions from $C(\mathcal{X})$ and $C'(\mathcal{Y})$ into $\mathcal{X}\ast \mathcal{Y}$. The projection $\mathcal{X}\ast \mathcal{Y}\rightarrow \mathcal{X}\ast \mathcal{Y}/(C(\mathcal{X})\vee C'(\mathcal{Y}))=\Sigma(\mathcal{X}\wedge \mathcal{Y}) $ is a motivic weak equivalence where $\Sigma (\mathcal{X}\wedge \mathcal{Y})$ is just the simplicial suspension $S^1\wedge \mathcal{X}\wedge \mathcal{Y}$. Let $f: X\times Y\rightarrow Z$ be a morphism of pointed motivic spaces. Then it induces a morphism from the reduced join $X\ast Y$ to $\Sigma Z$. This canonical weak equivalence can also be written as the following composite of the canonical projections \[\begin{tikzcd}
	{\mathcal{X}\ast \mathcal{Y}} & {\Sigma(\mathcal{X}\times \mathcal{Y})} & {\Sigma(\mathcal{X}\wedge \mathcal{Y})} & \cdot
	\arrow[from=1-1, to=1-2]
	\arrow[from=1-2, to=1-3]
\end{tikzcd}\]\\

The Hopf construction $H(f)$ of $f$ is the following composition in the pointed motivic homotopy category $\mathcal{H}_{\bullet}(\mathbb{Z})$ 
\[\begin{tikzcd}
	{\Sigma(X\wedge Y)} & {X\ast Y} & {\Sigma Z} 
	\arrow["\cong", from=1-1, to=1-2]
	\arrow[from=1-2, to=1-3]
\end{tikzcd}\]
where the first morphism in the composition is the inverse of the projection from $\mathcal{X}\ast \mathcal{Y}$ to $\Sigma(\mathcal{X}\wedge \mathcal{Y})$. The first motivic Hopf map $\eta_{1+(1)}$ is defined to be the Hopf construction of the pointed morphism $\mathbb{G}_{m}\times \mathbb{G}_{m}\rightarrow \mathbb{G}_{m}; (x,y)\mapsto x^{-1}y$.  In $\mathcal{H}_{\bullet}(\mathbb{Z})$ it can be identified with the canonical morphism $\mathbb{A}^2-{0}\rightarrow \mathbb{P}^1; (T_0, T_1)\mapsto [T_0:T_1]$. The second Hopf map $\nu_{2+(2)}$ is the Hopf construction of the multiplication map on $\mathrm{SL}_2$. The special linear group scheme can be identified with the motivic sphere $S^1\wedge \mathbb{G}_{m}^2$.\\

There is a $\mathbb{P}^1$-stabilization functor from the pointed unstable homotopy category $\mathcal{H}_{\bullet}(\mathbb{Z})$ to the stable motivic homotopy category $\mathcal{SH}(\mathbb{Z})$. Recall that maps between motivic spheres are always indexed by the bidegree of the target. Let $\eta$ denote the stablilization of $\eta_{1+(1)}$ and $\nu$ the stabilization of $\nu_{2+(2)}$. We also have the morphism  $\epsilon_{(1)}:\mathbb{G}_{m}\rightarrow \mathbb{G}_{m}; x\mapsto x^{-1}$. Then we define the hyperbolic plane to be $h_{1+(1)}: =1_{1+(1)}-\epsilon_{1+(1)}$ where $1_{1+(1)}$ is the identity map of $S^1\wedge\mathbb{G}_{m} $. Let $\epsilon$ denote the stabilization of $\epsilon_{(1)}$ and $h$ the stabilization of the hyperbolic plane. Morel proved in \cite{morelintro} the  relation $h\eta=\eta h=0$. In \cite{motihopf} Dugger and Isaksen proved the stable null-Hopf relation $\eta\nu=\nu\eta=0$. In the author's doctoral thesis \cite[Proposition 3.4.9]{dong2024motivictodabrackets} the unstable analogue of Morel's relation is proved. The author showed that the two composites $h_{1+(2)}\circ \eta_{1+(2)} $ and $\eta_{1+(2)}\circ h_{1+(3)}$ are $\mathbb{A}^1$-nullhomotopic. In this paper we want to show the unstable analogue of $\eta\nu=\nu\eta=0$. Precisely, we would like to show that the composites $\eta_{3+(2)}\circ \nu_{3+(3)}$ and $\nu_{3+(3)}\circ \eta _{4+(5)}$ are $\mathbb{A}^1$-nullhomotopic. Especially, these are the smallest bidegrees for which the composites are nullhomotopic.\\

Moreover we also prove the relation   $1_{1+(2)}+\eta_{1+(2)}\circ \Delta_{1+(3)}=-\epsilon_{1+(2)}$ which follows from the relation $q_{1+(1)}=1_{1+(1)}-\epsilon_{1+(1)}$. The morphism $q_{(1)}$ is the quadratic map $\mathbb{G}_{m}\rightarrow \mathbb{G}_{m}; x\mapsto x^{-1}$ and $\Delta_{(2)}$ is the diagonal morphism $\mathbb{G}_{m}\rightarrow\mathbb{G}_{m}^2; \ x\mapsto x\wedge x$. The result $q_{1+(1)}=1_{1+(1)}-\epsilon_{1+(1)}$ is again contained in the author's doctoral thesis \cite[Proposition 3.4.3]{dong2024motivictodabrackets}. The relation $1_{1+(2)}+\eta_{1+(2)}\circ \Delta_{1+(3)}=-\epsilon_{1+(2)}$ is the unstable analogue of the classical relation $1+\eta\rho= -\epsilon$ (cf. \cite[Lemma 6.3.4]{morelintro}) where $\rho$ is the stabilization of the diagonal morphism $\Delta_{(2)}$. Dugger and Isaksen also showed in \cite{motihopf} the relation $\epsilon\nu =-\nu$. Using $\epsilon\nu =-\nu$, $1+\eta\rho= -\epsilon$ and $\nu\eta=0$ we get $\nu^2=-\nu^2$. In this paper we show then $\nu_{3+(3)}=-\epsilon_{3+(3)}\circ \nu_{3+(3)}$ and $\nu_{3+(2)}\circ \nu_{4+(4)}= -\nu_{3+(2)}\circ \nu_{4+(4)}$. It follows from these identities that the composite $(1-\epsilon)_{3+(3)}\circ \nu_{3+(3)}\circ \nu_{4+(5)}$ is $\mathbb{A}^1$-nullhomotopic.\\

From the various $\mathbb{A}^1$-nullhomotopic composites we get several Toda brackets. We recall here briefly what Toda bracktes are. In \cite[Section 2.2]{dong2024motivictodabrackets} the author constrcuted Toda brackets in unstable motivic homotopy theory over the base scheme $\mathrm{Spec}\mathbb{Z}$. Suppose we are given a sequence of three composable morphisms of pointed motivic spaces \[\begin{tikzcd}
	{\mathcal{W}} & {\mathcal{X}} & {\mathcal{Y}} & {\mathcal{Z}}
	\arrow["\gamma", from=1-1, to=1-2]
	\arrow["\beta", from=1-2, to=1-3]
	\arrow["\alpha", from=1-3, to=1-4]
\end{tikzcd}\] such that the composites $\alpha\circ \beta$ and $\beta\circ \gamma$ are $\mathbb{A}^1$-nullhomotopic. Then we choose a nullhomotopy $A: C(\mathcal{X})=\Delta^1\wedge\mathcal{X}\rightarrow \mathcal{Z}$ for $\alpha\circ \beta$ and a nullhomotopy $B: C(\mathcal{W})=\Delta^1\wedge\mathcal{W}\rightarrow \mathcal{Y}$ for $\beta\circ \gamma$ where $\Delta^1$ is based at 1. In particular we get also the morphisms $\alpha\circ B$ and $A\circ C(\gamma)$, where $C(\gamma)$ is the morphism between the cones induced by $\gamma$. Hence we obtain a morphism \[\begin{tikzcd}
{\Sigma\mathcal{W}} & {C(\mathcal{W})\sqcup_{\mathcal{W}}C(\mathcal{W})} & {\mathcal{Z}}
\arrow["\sim", from=1-1, to=1-2]
\arrow[from=1-2, to=1-3]
\end{tikzcd}\] in the pointed motivic homotopy category. The identification of $\Sigma \mathcal{W}$ with the pushout $ C(\mathcal{W})\sqcup_{\mathcal{W}}C(\mathcal{W})$ is canonical. The Toda bracket $\{\alpha, \beta, \gamma \}\subset \mathcal{H}_{\bullet}(\mathbb{Z})(\Sigma\mathcal{W}, \mathcal{Z})$ is defined to be the set of all morphisms obtained in this way by choosing all possible nullhomotopies for $\alpha\circ \beta$ and $\beta\circ \gamma$.\\

More precisely we construct the Toda brackets $\{\eta_{3+(2)}, \nu_{3+(3)}, \eta _{4+(5)} \},  \\ \{\nu_{3+(3)}, \eta _{4+(5)}, \nu _{4+(6)} \}$ and $\{\eta_{3+(2)}, (1-\epsilon)_{3+(3)}, \nu_{3+(3)}\circ \nu_{4+(5)}\}$. Especially they are not trivial in the sense that they do not contain the homotopy classes of the constant morphisms.\\

We explain briefly the strategy we use  to prove that $\eta_{3+(2)}\circ \nu_{3+(3)}$ is nullhomotopic as the proof is rather technical. The proof is basically a variation of the methods given by Dugger and Isaksen in \cite{motihopf}. Their proof of $\eta\nu=0$ (cf.\cite[Proposition 5.4]{motihopf}) relies on a careful analysis of the morphisms in a commutative diagram and the fact that in ${\mathcal{SH}}(\mathbb{Z})$ the Hopf construction $H(f)$ of a pointed morphism $f: X\times Y\rightarrow Z$ is equal to
\[\begin{tikzcd}
	{X\wedge Y} & {X\times Y} & Z
	\arrow["\chi", from=1-1, to=1-2]
	\arrow["f", from=1-2, to=1-3]
\end{tikzcd}\]
where $\chi$ is a splitting map for the projection $X\times Y\rightarrow  X\wedge Y$. But in $\mathcal{H}_{\bullet}(\mathbb{Z})$ we have only $\Sigma^2H(f)= \Sigma^3f\circ \Sigma^2\chi$ where $\chi: \Sigma(X\wedge Y)\rightarrow \Sigma(X\times Y)$ is again the splitting map for the corresponding projection. Therefore the arguments by Dugger and Isaksen work only after taking two simplicial suspensions. We cannot apply them to $\eta_{3+(2)}\circ \nu_{3+(3)}$. It turns out that we have to work directly with joins.\\

\subsection*{Acknowledgement}
I would like to thank Tom Bachmann and Oliver R\"{o}ndigs for helpful comments.

\section{Technical lemmas}
\

In this section we prove some technical lemmas. Let $\mathbb{P}^1$ be the projective line over $\mathbb{Z}$. It is based at 1. In $\mathcal{H}_{\bullet}(\mathbb{Z})$ there is an isomorphism from $\mathbb{P}^1$ to $S^1\wedge \mathbb{G}_{m}$ (see \cite[p.408]{levine2010slices}).\\

\begin{lem}[{cf. \cite[Lemma 6.6.1]{morelintro}}] \label{lemma 2.1} In $\mathcal{H}_{\bullet}(\mathbb{Z})$ the permutation $\mathbb{P}^1\wedge \mathbb{P}^1\cong \mathbb{P}^1\wedge \mathbb{P}^1 $ corresponds to $-\epsilon_{1+(1)}\wedge \mathrm{id}_{1+(1)}$ if we identify $\mathbb{P}^1$ with $S^1\wedge \mathbb{G}_{m}$ via the isomorphism given in \cite[p.408]{levine2010slices}.\\
\end{lem}

\begin{proof}
From the elementary distinguished square for $\mathbb{P}^1$ we get an $\mathbb{A}^1$-weak equivalence from $\mathbb{P}^1/\mathbb{A}^1$ to $\mathbb{A}^1/\mathbb{G}_{m}$. The canonical projection from $\mathbb{P}^1$ to $\mathbb{P}^1/\mathbb{A}^1$ is also an $\mathbb{A}^1$-weak equivalence. Moreover we have the isomorphism of pointed motivic spaces $\mathbb{A}^1/\mathbb{G}_{m}\wedge \mathbb{A}^1/\mathbb{G}_{m}\cong \mathbb{A}^2/ \mathbb{A}^2-0$. Hence the permutation of the two copies of $\mathbb{P}^1$ corresponds in the homotopy category to the endomorphism $\phi$ of $\mathbb{A}^2/ \mathbb{A}^2-0$ which is induced by the ring homomorphism $\mathbb{Z}[T_0, T_1]\rightarrow \mathbb{Z}[T_0, T_1]; T_0\mapsto T_1, T_1\mapsto T_0$. Now let $\psi$ be the endomorphism of $\mathbb{A}^2/ \mathbb{A}^2-0$ which is induced by the ring homomorphism $\mathbb{Z}[T_0, T_1]\rightarrow \mathbb{Z}[T_0, T_1]; T_0\mapsto -T_0, T_1\mapsto T_1$. There is a sequence of $\mathbb{A}^1$-naive homotopies (see \cite{ASENS_2012_4_45_4_511_0}) from $\phi$ to $\psi$. The $\mathbb{A}^1$-homotopies are induced by the following ring homomorphisms. The first one is given by \begin{align*}
	\mathbb{Z}[T_0,T_1]\rightarrow	\mathbb{Z}[T_0,T_1,T]; \ T_0\mapsto T_1-TT_0,\  T_1\mapsto T_0\ \ \cdot
\end{align*}
The second one is given by \begin{align*}
	\mathbb{Z}[T_0,T_1]\rightarrow	\mathbb{Z}[T_0,T_1,T]; \ T_0\mapsto T_1-T_0,\  T_1\mapsto TT_0+(1-T)T_1\ \ \cdot
\end{align*}
Then the last one is given by \begin{align*}
	\mathbb{Z}[T_0,T_1]\rightarrow	\mathbb{Z}[T_0,T_1,T]; \ T_0\mapsto TT_1-T_0,\  T_1\mapsto T_1\ \ \cdot \\
\end{align*}
Therefore the permutation of $\mathbb{P}^1\wedge \mathbb{P}^1$ can be identified with the endomorphism $\psi$ of $\mathbb{A}^2/ \mathbb{A}^2-0$. On the other hand if we consider the zig-zag of $\mathbb{A}^1$-weak equivalences between $\mathbb{P}^1\wedge \mathbb{P}^1$ and $\mathbb{A}^2/ \mathbb{A}^2-0$ which we described at the beginning of the proof, then $\psi$ corresponds to $\tau\wedge\mathrm{id}_{\mathbb{P}^1}$. Here $\tau: \mathbb{P}^1\rightarrow \mathbb{P}^1$ is the map defined by \begin{align*}
	\mathbb{Z}[T_0,T_1]\rightarrow	\mathbb{Z}[T_0,T_1,T]; \ T_0\mapsto T_1,\  T_1\mapsto T_0\ \ \cdot 
\end{align*}
In particular $\tau $ corresponds to $-\epsilon_{1+(1)}$ under the isomorphism $\mathbb{P}^1\cong S^1\wedge \mathbb{G}_{m}$. Hence the permutation of $\mathbb{P}^1\wedge \mathbb{P}^1$ corresponds to $-\epsilon_{1+(1)}\wedge \mathrm{id}_{1+(1)}$.\\
\end{proof}

\begin{cor} \label{Corollary 2.2}
The morphism \begin{align*}
	S^1\wedge \mathbb{G}_{m}\wedge S^1\wedge \mathbb{G}_{m} \rightarrow		S^1\wedge \mathbb{G}_{m}\wedge S^1\wedge \mathbb{G}_{m}; s\wedge x \wedge t\wedge y\mapsto t\wedge y\wedge s\wedge x  
\end{align*}
is equal to $-\epsilon_{1+(1)}\wedge \mathrm{id}_{1+(1)}$ in $\mathcal{H}_{\bullet}(\mathbb{Z})$.\\
\end{cor}

\begin{cor} \label{Corollary 2.3}
The morphism \begin{align*}
S^1\wedge S^1\wedge \mathbb{G}_{m}\wedge\mathbb{G}_{m} \rightarrow		S^1\wedge S^1\wedge \mathbb{G}_{m}\wedge \mathbb{G}_{m}; s\wedge t \wedge x\wedge y\mapsto s\wedge t\wedge y\wedge x
\end{align*}
is equal to $\epsilon_{2+(2)}=\mathrm{id}_{	S^1\wedge S^1\wedge \mathbb{G}_{m}}\wedge \epsilon_{(1)}$. The identity $\epsilon_{1+(2)}=\mathrm{id}_{	S^1\wedge \mathbb{G}_{m}}\wedge \epsilon_{(1)}$ is proved in the author's doctoral thesis \cite[Lemma 3.4.7]{dong2024motivictodabrackets}.\\
\end{cor}

\begin{proof}
Let $\theta$ denote the isomorphism \begin{align*}
S^1\wedge S^1\wedge \mathbb{G}_{m}\wedge\mathbb{G}_{m} \rightarrow		S^1\wedge \mathbb{G}_{m}\wedge S^1\wedge \mathbb{G}_{m}; s\wedge t \wedge x\wedge y\mapsto s\wedge x\wedge t\wedge y 
\end{align*}
We denote the morphism in Corollary 2.2 by $\lambda$. Then the composite $\theta^{-1}\circ \lambda\circ \theta$ is the morphism \begin{align*}
	S^1\wedge S^1\wedge \mathbb{G}_{m}\wedge\mathbb{G}_{m} \rightarrow		S^1\wedge S^1\wedge \mathbb{G}_{m}\wedge\mathbb{G}_{m}; s\wedge t \wedge x\wedge y\mapsto t\wedge s\wedge y\wedge x \ \ \cdot
\end{align*}
This is exactly the inverse of \begin{align*}
	S^1\wedge S^1\wedge \mathbb{G}_{m}\wedge\mathbb{G}_{m} \rightarrow		S^1\wedge S^1\wedge \mathbb{G}_{m}\wedge \mathbb{G}_{m}; s\wedge t \wedge x\wedge y\mapsto s\wedge t\wedge y\wedge x \ \ \cdot
\end{align*}
By Corollary 2.3 the composite $\theta^{-1}\circ \lambda\circ \theta$ is also equal to $-(\theta^{-1}\circ \epsilon_{1+(1)}\wedge \mathrm{id}_{1+(1)} \circ \theta )$. Furthermore $\theta^{-1}\circ \epsilon_{1+(1)}\wedge \mathrm{id}_{1+(1)} \circ \theta$ is just $\epsilon_{2+(2)}$. Therefore the morphism \begin{align*}
	S^1\wedge S^1\wedge \mathbb{G}_{m}\wedge\mathbb{G}_{m} \rightarrow		S^1\wedge S^1\wedge \mathbb{G}_{m}\wedge \mathbb{G}_{m}; s\wedge t \wedge x\wedge y\mapsto s\wedge t\wedge y\wedge x \ \ \cdot
\end{align*}
equals to $\epsilon_{2+(2)}$.\\
\end{proof}

\begin{cor}\label{Corollary 2.4} The endomorphism of $S^1\wedge S^1\wedge \mathbb{G}_{m}\wedge\mathbb{G}_{m}\wedge\mathbb{G}_{m} $ which is given by \begin{align*}
		s\wedge t \wedge x'\wedge x\wedge y \mapsto s\wedge t \wedge x'\wedge y\wedge x
	\end{align*}
is equal to the endomorphism $S^1\wedge S^1\wedge \mathbb{G}_{m}\wedge\mathbb{G}_{m}\wedge\mathbb{G}_{m} $ given by \begin{align*}
	s\wedge t \wedge x'\wedge x\wedge y \mapsto s\wedge t \wedge x'\wedge x^{-1}\wedge y \ \ \cdot
\end{align*} 
\end{cor}
\

The first motivic Hopf map $\eta_{1+(1)}$ is defined to be the Hopf construction of the pointed morphism $\mathbb{G}_{m}\times \mathbb{G}_{m}\rightarrow \mathbb{G}_{m}; (x,y)\mapsto x^{-1}y$. In $\mathcal{H}_{\bullet}(\mathbb{Z})$ we have the following relation.\\

\begin{lem}[{cf. \cite[Proposition 4.10]{motihopf}}] \label{Lemma 2.5}In $\mathcal{H}_{\bullet}(\mathbb{Z})$ the map $\eta_{1+(2)}$ is equal to the $\mathbb{G}_{m}$-suspension of the Hopf construction on the pointed morphism $\mu:\mathbb{G}_{m}\times \mathbb{G}_{m}\rightarrow \mathbb{G}_{m}; (x,y)\mapsto xy$.\\
\end{lem}

\begin{proof}
We have the commutative diagram \[\begin{tikzcd}
	{\Sigma(\mathbb{G}_{m}\wedge \mathbb{G}_{m})} & {\mathbb{G}_{m}\ast\mathbb{G}_{m}} & {\Sigma\mathbb{G}_{m}} \\
	{\Sigma(\mathbb{G}_{m}\wedge \mathbb{G}_{m})} & {\mathbb{G}_{m}\ast\mathbb{G}_{m}} & {\Sigma\mathbb{G}_{m}}
	\arrow["\sim", from=1-1, to=1-2]
	\arrow["{\epsilon_{1+(2)}}"', from=1-1, to=2-1]
	\arrow["f", from=1-2, to=1-3]
	\arrow["{\epsilon_{(1)}\ast\mathrm{id}}"', from=1-2, to=2-2]
	\arrow["{\mathrm{id}}"', from=1-3, to=2-3]
	\arrow["\sim", from=2-1, to=2-2]
	\arrow["\tilde{\mu}", from=2-2, to=2-3]
\end{tikzcd}\]
in $\mathcal{H}_{\bullet}(\mathbb{Z})$  where $f$ is induced by $(x,y)\mapsto x^{-1}y$ and $\tilde{\mu}$ induced by $(x,y)\mapsto xy$. The composition along the top of the diagram is the Hopf construction $H(f)$. The composition along the bottom of the diagram is $\eta_{1+(1)}$. Therefore we have $\eta_{1+(1)}\circ \epsilon_{1+(2)}= H(f)$. In the author's doctoral thesis \cite[Proposition 3.4.9 ]{dong2024motivictodabrackets} the relation $\eta_{1+(2)}\circ h_{1+(3)}=0$ is proved where $h_{1+(1)}: =1_{1+(1)}-\epsilon_{1+(1)}$. Therefore $H(f)\wedge \mathrm{id}_{\mathbb{G}_{m}}=\eta_{1+(2)}\circ \epsilon_{1+(3)}= \eta_{1+(2)}$.\\
	\end{proof}
	
Let $\mathrm{SL}_{2}= \mathrm{Spec} \ \mathbb{Z}[T_{11}, T_{12}, T_{21}, T_{22}]/(\mathrm{det}-1)$ be the special linear group scheme. It is equipped with the base point $\begin{pmatrix}
	1&0\\
	0&1
\end{pmatrix}$. Then the projection onto the last column from $\mathrm{SL}_{2}$ to $\mathbb{A}^2-0$ is a pointed $\mathbb{A}^1$-weak equivalence if we equip $\mathbb{A}^2-0$ with the base point $(0,1)$ (see \cite[Example 2.12(3)]{motihopf}). Now there is a pointed morphism $\alpha: \mathrm{SL}_{2}\times \mathbb{G}_{m} \rightarrow \mathrm{SL}_{2} $ given by the ring homomorphism \begin{align*}
T_{11}\mapsto T_{11}, T_{12}\mapsto tT_{12}, T_{21}\mapsto t^{-1}T_{21}, T_{22}\mapsto T_{22} \ \ \cdot
\end{align*}
We abuse here the notation and denote this map in the following by \begin{align*}
(\begin{pmatrix}
		T_{11}&T_{12}\\
		T_{21}&T_{22}
	\end{pmatrix}, t) \mapsto \begin{pmatrix}
	T_{11}&tT_{12}\\
	t^{-1}T_{21}&T_{22}
	\end{pmatrix} \ \ \cdot
\end{align*}
There is a canonical elementary distinguished square \[\begin{tikzcd}
	{\mathbb{G}_{m}\times \mathbb{G}_{m}} & {\mathbb{A}^1\times \mathbb{G}_{m}} \\
	{\mathbb{G}_{m}\times \mathbb{A}^1} & {\mathbb{A}^2-0} & \cdot
	\arrow[hook, from=1-1, to=1-2]
	\arrow[hook, from=1-1, to=2-1]
	\arrow[hook, from=1-2, to=2-2]
	\arrow[hook, from=2-1, to=2-2]
\end{tikzcd}\]
Via this elementary distinguished square we can identify $\mathbb{A}^2-0$ with $S^1\wedge \mathbb{G}_{m}^2$. Hence $\mathrm{SL}_{2}$ is also identified with $S^1\wedge \mathbb{G}_{m}^2$. Via these identifications the Hopf construction $H(\alpha)$ can be considered as a morphism $S^2\wedge\mathbb{G}_{m}^3 \rightarrow S^2\wedge\mathbb{G}_{m}^2$. In \cite[Lemma 5.3]{motihopf} Dugger and Isaksen showed that stably the Hopf construction $H(\alpha)$ is equal to $\eta$. We prove now an unstable analogue of this claim.\\

\begin{prop}\label{Propostion 2.6}
	The Hopf construction $H(\alpha)$ is equal to $\eta_{2+(2)}$ in $\mathcal{H}_{\bullet}(\mathbb{Z})$.\\
\end{prop}

\begin{proof}
	We first explain explicitly how to identify $H(\alpha)$ with a morphism of motivic spheres. The morphism $\alpha$ corresponds under the projection onto the last column from $\mathrm{SL}_{2}$ to $\mathbb{A}^2-0$ to a pointed morphism $\alpha': \mathbb{A}^2-0\times \mathbb{G}_{m} \rightarrow \mathbb{A}^2-0 $ which is given by $((a,b),t)\mapsto (ta,b)$.\\
	
	Let $\chi$ denote the pushout of the diagram \[\begin{tikzcd}
		{\mathbb{G}_{m}\times \mathbb{G}_{m}} & {\mathbb{A}^1\times \mathbb{G}_{m}} \\
		{\mathbb{G}_{m}\times \mathbb{A}^1} & \cdot
		\arrow[hook, from=1-1, to=1-2]
		\arrow[hook, from=1-1, to=2-1]
	\end{tikzcd}\]
	The canonical map from $\chi$ to $\mathbb{A}^2-0$ is a pointed $\mathbb{A}^1$-weak equivalence, if we equip $\chi$ with the base point $(0,1)$ coming from $\mathbb{A}^1\times \mathbb{G}_{m}$. In particular, the morphism $\alpha'$ corresponds then to a pointed morphism $\alpha'': \chi \times \mathbb{G}_{m}\rightarrow \chi$ which is given by the same formula as for $\alpha'$. The Hopf construction yields a morphism $H(\alpha''): \Sigma(\chi\wedge \mathbb{G}_{m})\rightarrow \Sigma\chi$.\\
	
	Furthermore we also have the following commutative diagram \[\begin{tikzcd}
		{\mathbb{A}^1\times \mathbb{G}_{m}} & {\mathbb{G}_{m}\times \mathbb{G}_{m}} & {\mathbb{G}_{m}\times \mathbb{A}^1} \\
		{\mathbb{A}^1\wedge \mathbb{A}^1} & {\mathbb{G}_{m}\wedge \mathbb{G}_{m}} & {\mathbb{A}^1\wedge \mathbb{A}^1} \\
		\ast & {\mathbb{G}_{m}\wedge \mathbb{G}_{m}} & {C(\mathbb{A}^1\wedge \mathbb{A}^1)} \\
		\ast & {\mathbb{G}_{m}\wedge \mathbb{G}_{m}} & {C(\mathbb{G}_{m}\wedge \mathbb{G}_{m})}
		\arrow[from=1-1, to=2-1]
		\arrow[hook', from=1-2, to=1-1]
		\arrow[hook, from=1-2, to=1-3]
		\arrow[from=1-2, to=2-2]
		\arrow[from=1-3, to=2-3]
		\arrow[from=2-1, to=3-1]
		\arrow[hook', from=2-2, to=2-1]
		\arrow[hook, from=2-2, to=2-3]
		\arrow["{\mathrm{id}}"', from=2-2, to=3-2]
		\arrow[hook, from=2-3, to=3-3]
		\arrow[from=3-2, to=3-1]
		\arrow[hook, from=3-2, to=3-3]
		\arrow[from=4-1, to=3-1]
		\arrow["{\mathrm{id}}", from=4-2, to=3-2]
		\arrow[from=4-2, to=4-1]
		\arrow[hook, from=4-2, to=4-3]
		\arrow[hook', from=4-3, to=3-3]
	\end{tikzcd}\]
	where $C(\mathbb{A}^1\wedge \mathbb{A}^1)$ is the cone on $\mathbb{A}^1\wedge \mathbb{A}^1$ and $C(\mathbb{G}_{m}\wedge \mathbb{G}_{m})$ the cone on $\mathbb{G}_{m}\wedge \mathbb{G}_{m}$. This commutative diagram induces a zig-zag of pointed $\mathbb{A}^1$-weak equivalences between $\chi$ and $S^1\wedge \mathbb{G}_{m}^2$. Using these weak equivalences we get from $H(\alpha'')$ a morphism $S^2\wedge\mathbb{G}_{m}^3\rightarrow S^2\wedge \mathbb{G}_{m}^2$. Now we show that this morphism is equal to $\eta_{2+(2)}$. By Lemma 2.5 we have $\mathrm{id}_{S^1}\wedge\mu\wedge \mathrm{id}_{\mathbb{G}_{m}}$ is $\eta_{2+(2)}$.\\
	
	Let $\mathcal{Y}$ denote the pushout of \[\begin{tikzcd}
		{\mathbb{G}_{m}\wedge \mathbb{G}_{m}} & {\mathbb{A}^1\wedge \mathbb{A}^1} \\
		{\mathbb{A}^1\wedge \mathbb{A}^1}
		\arrow[hook, from=1-1, to=1-2]
		\arrow[hook, from=1-1, to=2-1]
	\end{tikzcd}\] and $\mathcal{Y'}$ the pushout of \[\begin{tikzcd}
	{\mathbb{G}_{m}\wedge \mathbb{G}_{m}} & {C(\mathbb{A}^1\wedge \mathbb{A}^1)} \\
	\ast & \cdot
	\arrow[hook, from=1-1, to=1-2]
	\arrow[hook, from=1-1, to=2-1]
	\end{tikzcd}\]
Then we only need to show that the following diagram (1) in $\mathcal{H}_{\bullet}(\mathbb{Z})$ is commutative \[\begin{tikzcd}
	{\mathcal{X}\ast \mathbb{G}_{m}} & {\Sigma(\mathcal{X}\wedge \mathbb{G}_{m})} & {\Sigma(\mathcal{Y}\wedge \mathbb{G}_{m})} & {\Sigma(\mathcal{Y}'\wedge \mathbb{G}_{m})} & {S^2\wedge \mathbb{G}_{m}^3} \\
	{\Sigma\mathcal{X}} &&&& {S^2\wedge \mathbb{G}_{m}^3} \\
	{\Sigma\mathcal{Y}} & {\Sigma\mathcal{Y}'} &&& {S^2\wedge \mathbb{G}_{m}^2}
	\arrow["\sim", from=1-1, to=1-2]
	\arrow["{H(\alpha'')}"', from=1-1, to=2-1]
	\arrow["\sim", from=1-2, to=1-3]
	\arrow["\sim", from=1-3, to=1-4]
	\arrow["\sim"', from=1-5, to=1-4]
	\arrow["{\mathrm{id}_{S^1}\wedge \mathrm{id}_{S^1}\wedge \tau'}"', from=1-5, to=2-5]
	\arrow["\sim"', from=2-1, to=3-1]
	\arrow["{\mathrm{id}_{S^1}\wedge H(\mu)\wedge \mathrm{id}_{\mathbb{G}_{m}}}"', from=2-5, to=3-5]
	\arrow["\sim", from=3-1, to=3-2]
	\arrow["\sim"', from=3-5, to=3-2]
\end{tikzcd}\] where $\tau': \mathbb{G}_{m}^3\rightarrow \mathbb{G}_{m}^3: x\wedge y\wedge z\mapsto z\wedge x\wedge y$ and the weak equivalences which are induced by the weak equivalences between the pushouts are indicated by $\sim$.\\ 

We first observe that $\mathrm{id}_{S^1}\wedge \mathrm{id}_{S^1}\wedge \tau'$ is actually equal to the identity. The morphism $\mathrm{id}_{S^1}\wedge \mathrm{id}_{S^1}\wedge \tau'$ can be written as the composite of the following two morphisms
\begin{align*}
g: S^2\wedge \mathbb{G}_{m}^3\rightarrow S^2\wedge \mathbb{G}_{m}^3; 	s\wedge t\wedge x\wedge y\wedge z\mapsto s\wedge t\wedge x\wedge z\wedge y 
\end{align*} and
\begin{align*}
h: S^2\wedge \mathbb{G}_{m}^3\rightarrow S^2\wedge \mathbb{G}_{m}^3; 	s\wedge t\wedge x\wedge y\wedge z\mapsto s\wedge t\wedge y\wedge x\wedge y \ \ \cdot
\end{align*} 
So we have $\mathrm{id}_{S^1}\wedge \mathrm{id}_{S^1}\wedge \tau'=h\circ g$. By Corollary 2.3 the morphism $g$ equals to $\mathrm{id}_{S^1\wedge S^1\wedge \mathbb{G}_{m}}\wedge \epsilon_{(1)}\wedge \mathrm{id}_ {\mathbb{G}_{m}}$. Again by Corollary 2.3 we have $h= \epsilon_{2+(3)}$ where $\epsilon_{2+(3)}$ is given by $	s\wedge t\wedge x\wedge y\wedge z\mapsto s\wedge t\wedge x^{-1}\wedge y\wedge z $. Since $\epsilon_{1+(2)}=\mathrm{id}_{	S^1\wedge \mathbb{G}_{m}}\wedge \epsilon_{(1)}$, we get  $\mathrm{id}_{S^1\wedge S^1\wedge \mathbb{G}_{m}}\wedge \epsilon_{(1)}\wedge \mathrm{id}_ {\mathbb{G}_{m}}= \epsilon_{2+(3)}$. Hence $\mathrm{id}_{S^1}\wedge \mathrm{id}_{S^1}\wedge \tau'$ is equal to the composite $\epsilon_{2+(3)}\circ \epsilon_{2+(3)} $ which is the identity.\\

Recall that the weak equivalence $\mathcal{X}\ast\mathbb{G}_{m}\rightarrow \Sigma(\mathcal{X}\wedge \mathbb{G}_{m})$ can be written as the composite \[\begin{tikzcd}
	{\mathcal{X}\ast \mathbb{G}_{m}} & {\Sigma(\mathcal{X}\times \mathbb{G}_{m})} & {\Sigma(\mathcal{X}\wedge \mathbb{G}_{m})} & \cdot
	\arrow[from=1-1, to=1-2]
	\arrow[from=1-2, to=1-3]
\end{tikzcd}\] In particular there is a pointed morphism $\alpha''': \mathcal{Y'}\times \mathbb{G}_{m} \rightarrow \mathcal{Y'}$ which is also given by $((a,b),t)\mapsto (ta,b)$. Therefore the diagram \[\begin{tikzcd}
{\mathcal{X}\ast \mathbb{G}_{m}} & {\Sigma(\mathcal{X}\wedge \mathbb{G}_{m})} & {\Sigma(\mathcal{Y}\wedge \mathbb{G}_{m})} & {\Sigma(\mathcal{Y}'\wedge \mathbb{G}_{m})} \\
{\Sigma\mathcal{X}} \\
{\Sigma\mathcal{Y}} & {\Sigma\mathcal{Y}'} && {\mathcal{Y}'\ast \mathbb{G}_{m}}
\arrow["\sim", from=1-1, to=1-2]
\arrow["{{H(\alpha'')}}"', from=1-1, to=2-1]
\arrow["\sim"{description}, from=1-1, to=3-4]
\arrow["\sim", from=1-2, to=1-3]
\arrow["\sim", from=1-3, to=1-4]
\arrow["\sim"', from=2-1, to=3-1]
\arrow["\sim", from=3-1, to=3-2]
\arrow["\sim"', from=3-4, to=1-4]
\arrow["{H(\alpha''')}"{description}, from=3-4, to=3-2]
\end{tikzcd}\] commutes where $\mathcal{X}\ast \mathbb{G}_{m}\rightarrow \mathcal{Y}'\ast \mathbb{G}_{m}$ is the weak equivalence induced by the weak equivalence from $\mathcal{X}$ to $\mathcal{Y'}$.\\ 

Especially we get \[\begin{tikzcd}
{\mathcal{X}\ast \mathbb{G}_{m}} & {\Sigma(\mathcal{X}\wedge \mathbb{G}_{m})} & {\Sigma(\mathcal{Y}\wedge \mathbb{G}_{m})} & {\Sigma(\mathcal{Y}'\wedge \mathbb{G}_{m})} & {S^2\wedge \mathbb{G}_{m}^3} \\
{\Sigma\mathcal{X}} && {\mathcal{Y}'\ast \mathbb{G}_{m}} && {S^2\wedge \mathbb{G}_{m}^3} \\
{\Sigma\mathcal{Y}} & {\Sigma\mathcal{Y}'} &&& {S^2\wedge \mathbb{G}_{m}^2} & \cdot
\arrow["\sim", from=1-1, to=1-2]
\arrow["{{H(\alpha'')}}"', from=1-1, to=2-1]
\arrow["\sim"{description}, from=1-1, to=2-3]
\arrow["\sim", from=1-2, to=1-3]
\arrow["\sim", from=1-3, to=1-4]
\arrow["\sim"', from=1-5, to=1-4]
\arrow["{{\mathrm{id}_{S^1}\wedge \mathrm{id}_{S^1}\wedge \tau'}}"', from=1-5, to=2-5]
\arrow["\sim"', from=2-1, to=3-1]
\arrow["\sim"{description}, from=2-3, to=1-4]
\arrow["{H(\alpha''')}"{description}, from=2-3, to=3-2]
\arrow["{{\mathrm{id}_{S^1}\wedge H(\mu)\wedge \mathrm{id}_{\mathbb{G}_{m}}}}"', from=2-5, to=3-5]
\arrow["\sim", from=3-1, to=3-2]
\arrow["\sim"', from=3-5, to=3-2]
\end{tikzcd}\] It follows now that we only have to show that the diagram \[\begin{tikzcd}
{\mathcal{Y}'\ast \mathbb{G}_{m}} && {\Sigma(\mathcal{Y}'\wedge \mathbb{G}_{m})} & {S^2\wedge \mathbb{G}_{m}^3} \\
&&& {S^2\wedge \mathbb{G}_{m}^3} \\
{\Sigma\mathcal{Y}'} &&& {S^2\wedge \mathbb{G}_{m}^2} & \cdot
\arrow["\sim", from=1-1, to=1-3]
\arrow["{H(\alpha''')}"', from=1-1, to=3-1]
\arrow["\sim"', from=1-4, to=1-3]
\arrow["{{\mathrm{id}_{S^1}\wedge \mathrm{id}_{S^1}\wedge \tau'}}"', from=1-4, to=2-4]
\arrow["{{\mathrm{id}_{S^1}\wedge H(\mu)\wedge \mathrm{id}_{\mathbb{G}_{m}}}}"', from=2-4, to=3-4]
\arrow["\sim"', from=3-4, to=3-1]
\end{tikzcd}\] is commutative in $\mathcal{H}_{\bullet}(\mathbb{Z})$.\\

We first analyse the composite $\mathrm{id}_{S^1}\wedge H(\mu)\wedge \mathrm{id}_{\mathbb{G}_{m}}\circ \mathrm{id}_{S^1}\wedge \mathrm{id}_{S^1}\wedge \tau'$. Recall that $H(\mu)$ is the composite \[\begin{tikzcd}
	{\Sigma(\mathbb{G}_{m}\wedge \mathbb{G}_{m})} & {\mathbb{G}_{m}\ast\mathbb{G}_{m}} & {\Sigma\mathbb{G}_{m}}
	\arrow["\sim", from=1-1, to=1-2]
	\arrow["{\tilde{\mu}}", from=1-2, to=1-3]
\end{tikzcd}\] where $\tilde{\mu}$ is the morphism induced by $\mu$. Hence we can consider the following diagram \[\begin{tikzcd}
{S^1\wedge(\mathbb{G}_{m}\ast \mathbb{G}_{m})\wedge \mathbb{G}_{m}} & {S^2\wedge \mathbb{G}_{m}^3} \\
{S^2\wedge \mathbb{G}_{m}^2} & {S^2\wedge \mathbb{G}_{m}^3} \\
{S^1\wedge(\mathbb{G}_{m}\ast \mathbb{G}_{m})} & {(\mathbb{G}_{m}\ast \mathbb{G}_{m})\ast \mathbb{G}_{m}}
\arrow["\sim", from=1-1, to=1-2]
\arrow["{\mathrm{id}_{S^1}\wedge \tilde{\mu}\wedge \mathrm{id}_{\mathbb{G}_{m}}}"', from=1-1, to=2-1]
\arrow["{\mathrm{id}_{S^1}\wedge \mathrm{id}_{S^1}\wedge \tau'}"', from=2-2, to=1-2]
\arrow["\sim", from=3-1, to=2-1]
\arrow["\sim"', from=3-2, to=2-2]
\arrow["{\mu'}", from=3-2, to=3-1]
\end{tikzcd}\]
The morphism $\mu'$ is induced by the pointed morphism $(\mathbb{G}_{m}\ast \mathbb{G}_{m})\times \mathbb{G}_{m}\rightarrow \mathbb{G}_{m}\ast \mathbb{G}_{m}; (\overline{(t,x,y)},z)\mapsto \overline{(t,zx,y)}$. In the next step we have to use the geometric realization functor $|\cdot|$ which the author developed in her doctoral thesis \cite[Section 2]{dong2024motivictodabrackets}. Especially, there is a sectionwise weak equivalence $\sigma: |(\mathbb{G}_{m}\ast \mathbb{G}_{m})\ast \mathbb{G}_{m}|\rightarrow |S^1\wedge(\mathbb{G}_{m}\ast \mathbb{G}_{m})\wedge \mathbb{G}_{m}|$ such that the diagram \[\begin{tikzcd}
	{|S^1\wedge(\mathbb{G}_{m}\ast \mathbb{G}_{m})\wedge \mathbb{G}_{m}|} & {|S^2\wedge \mathbb{G}_{m}^3|} \\
	{|S^2\wedge \mathbb{G}_{m}^2|} & {|S^2\wedge \mathbb{G}_{m}^3|} \\
	{|S^1\wedge(\mathbb{G}_{m}\ast \mathbb{G}_{m})|} & {|(\mathbb{G}_{m}\ast \mathbb{G}_{m})\ast \mathbb{G}_{m}|}
	\arrow["\sim", from=1-1, to=1-2]
	\arrow["{|\mathrm{id}_{S^1}\wedge \tilde{\mu}\wedge \mathrm{id}_{\mathbb{G}_{m}}|}"', from=1-1, to=2-1]
	\arrow["{|\mathrm{id}_{S^1}\wedge \mathrm{id}_{S^1}\wedge \tau'|}"', from=2-2, to=1-2]
	\arrow["\sim", from=3-1, to=2-1]
	\arrow["\sigma"{description}, from=3-2, to=1-1]
	\arrow["\sim"', from=3-2, to=2-2]
	\arrow["{|\mu'|}", from=3-2, to=3-1]
\end{tikzcd}\] is commutative in the corresponding homotopy category. The weak equivalence $\sigma$ is given by $\overline{(c,(\overline{b,x,y}),z)}\mapsto \begin{cases}
(1-c)b\wedge \overline{(\frac{(1-c)(1-b)}{1-(1-c)b},z,x)}\wedge y  & \text{if} \  c\neq 0 \ \text {or}  \ b\neq 1\ \\
\ast & \text{if} \ c= 0 \ \text {and}  \ b=1\
\end{cases}
$ for all $x,y,z\in  \mathbb{G}_{m}$.\\

Now the pointed continuous map \begin{align*}
I/\partial I\wedge I/\partial I \rightarrow I/\partial I\wedge I/\partial I; c\wedge b\mapsto \begin{cases}
		(1-c)b\wedge\frac{(1-c)(1-b)}{1-(1-c)b}  & \text{if} \  c\neq 0 \ \text {or}  \ b\neq 1\ \\
\ast & \text{if} \ c= 0 \ \text {and}  \ b=1\
	\end{cases}
\end{align*}
is pointed homotopic to the identity. Hence the diagram \[\begin{tikzcd}
	{|S^1\wedge(\mathbb{G}_{m}\ast \mathbb{G}_{m})\wedge \mathbb{G}_{m}|} & {|S^2\wedge \mathbb{G}_{m}^3|} \\
	{|S^2\wedge \mathbb{G}_{m}^2|} & {|S^2\wedge \mathbb{G}_{m}^3|} \\
	{|S^1\wedge(\mathbb{G}_{m}\ast \mathbb{G}_{m})|} & {|(\mathbb{G}_{m}\ast \mathbb{G}_{m})\ast \mathbb{G}_{m}|}
	\arrow["\sim", from=1-1, to=1-2]
	\arrow["{|\mathrm{id}_{S^1}\wedge \tilde{\mu}\wedge \mathrm{id}_{\mathbb{G}_{m}}|}"', from=1-1, to=2-1]
	\arrow["{|\mathrm{id}_{S^1}\wedge \mathrm{id}_{S^1}\wedge \tau'|}"', from=2-2, to=1-2]
	\arrow["\sim", from=3-1, to=2-1]
	\arrow["\sigma"{description}, from=3-2, to=1-1]
	\arrow["\sim"', from=3-2, to=2-2]
	\arrow["{|\mu'|}", from=3-2, to=3-1]
\end{tikzcd}\] is commutative in the corresponding homotopy category. It follows that also the diagram \[\begin{tikzcd}
{S^1\wedge(\mathbb{G}_{m}\ast \mathbb{G}_{m})\wedge \mathbb{G}_{m}} & {S^2\wedge \mathbb{G}_{m}^3} \\
{S^2\wedge \mathbb{G}_{m}^2} & {S^2\wedge \mathbb{G}_{m}^3} \\
{S^1\wedge(\mathbb{G}_{m}\ast \mathbb{G}_{m})} & {(\mathbb{G}_{m}\ast \mathbb{G}_{m})\ast \mathbb{G}_{m}}
\arrow["{{\mathrm{id}_{S^1}\wedge \tilde{\mu}\wedge \mathrm{id}_{\mathbb{G}_{m}}}}"', from=1-1, to=2-1]
\arrow["\sim"', from=1-2, to=1-1]
\arrow["{{\mathrm{id}_{S^1}\wedge \mathrm{id}_{S^1}\wedge \tau'}}"', from=2-2, to=1-2]
\arrow["\sim", from=2-2, to=3-2]
\arrow["\sim", from=3-1, to=2-1]
\arrow["{{\mu'}}", from=3-2, to=3-1]
\end{tikzcd}\] commutes in $\mathcal{H}_{\bullet}(\mathbb{Z})$. One concludes that the composite $\mathrm{id}_{S^1}\wedge H(\mu)\wedge \mathrm{id}_{\mathbb{G}_{m}}\circ \mathrm{id}_{S^1}\wedge \mathrm{id}_{S^1}\wedge \tau'$ can be written as the composite \[\begin{tikzcd}
{S^2\wedge \mathbb{G}_{m}^3} & {(\mathbb{G}_{m}\ast \mathbb{G}_{m})\ast \mathbb{G}_{m}} & {S^1\wedge(\mathbb{G}_{m}\ast \mathbb{G}_{m})} & {S^2\wedge \mathbb{G}_{m}^2} & \cdot
\arrow["\sim", from=1-1, to=1-2]
\arrow["{\mu'}", from=1-2, to=1-3]
\arrow["\sim", from=1-3, to=1-4]
\end{tikzcd}\]
Altogether we get the diagram \[\begin{tikzcd}
	{\mathcal{Y}'\ast \mathbb{G}_{m}} && {\Sigma(\mathcal{Y}'\wedge \mathbb{G}_{m})} & {S^2\wedge \mathbb{G}_{m}^3} \\
	&&& {(\mathbb{G}_{m}\ast \mathbb{G}_{m})\ast \mathbb{G}_{m}} \\
	&&& {S^1\wedge(\mathbb{G}_{m}\ast \mathbb{G}_{m})} \\
	{\Sigma\mathcal{Y}'} &&& {S^2\wedge \mathbb{G}_{m}^2} & \cdot
	\arrow["\sim", from=1-1, to=1-3]
	\arrow["{H(\alpha''')}"', from=1-1, to=4-1]
	\arrow["\sim"', from=1-4, to=1-3]
	\arrow["\sim"', from=2-4, to=1-4]
	\arrow["{\mu'}", from=2-4, to=3-4]
	\arrow[from=3-4, to=4-4]
	\arrow["\sim", from=4-4, to=4-1]
\end{tikzcd}\]
Note that there is now also a weak equivalence $(S^1\wedge\mathbb{G}_{m}^2)\ast \mathbb{G}_{m}\rightarrow \mathcal{Y}'\ast \mathbb{G}_{m} $ which is induced by the weak equivalnce $S^1\wedge\mathbb{G}_{m}^2\rightarrow \mathcal{Y}'$. Hence we can consider the following diagram \[\begin{tikzcd}
	{\mathcal{Y}'\ast \mathbb{G}_{m}} && {\Sigma(\mathcal{Y}'\wedge \mathbb{G}_{m})} & {S^2\wedge \mathbb{G}_{m}^3} \\
	& {(S^1\wedge\mathbb{G}_{m}^2)\ast \mathbb{G}_{m}} && {(\mathbb{G}_{m}\ast \mathbb{G}_{m})\ast \mathbb{G}_{m}} \\
	&&& {S^1\wedge(\mathbb{G}_{m}\ast \mathbb{G}_{m})} \\
	{\Sigma\mathcal{Y}'} &&& {S^2\wedge \mathbb{G}_{m}^2} & \cdot
	\arrow["\sim", from=1-1, to=1-3]
	\arrow["{H(\alpha''')}"', from=1-1, to=4-1]
	\arrow["\sim"', from=1-4, to=1-3]
	\arrow["\sim", from=2-2, to=1-1]
	\arrow["\sim"', from=2-4, to=1-4]
	\arrow["\sim", from=2-4, to=2-2]
	\arrow["{\mu'}", from=2-4, to=3-4]
	\arrow[from=3-4, to=4-4]
	\arrow["\sim", from=4-4, to=4-1]
\end{tikzcd}\]
By inspection we see directly that the two inner diagrams are commutative. Thus the whole diagram commutes and the claim follows.\\
\end{proof}

Let $\Delta_{(2)}:\mathbb{G}_{m}\rightarrow \mathbb{G}_{m}\wedge \mathbb{G}_{m}; x\mapsto x\wedge x $ be the diagonal morphism. Let $\pi_1: \mathbb{G}_{m}\times\mathbb{G}_{m}\rightarrow \mathbb{G}_{m}$ be the projection onto the first factor and $\pi_2: \mathbb{G}_{m}\times\mathbb{G}_{m}\rightarrow \mathbb{G}_{m}$ the projection onto the second factor. Let $p:\mathbb{G}_{m}\times\mathbb{G}_{m}\rightarrow \mathbb{G}_{m}\wedge\mathbb{G}_{m} $ be the canonical projection.\\

Let $i_1:\Sigma \mathbb{G}_{m}\rightarrow \Sigma\mathbb{G}_{m}\vee \Sigma\mathbb{G}_{m}\vee\Sigma(\mathbb{G}_{m}\wedge\mathbb{G}_{m}) $ be the inclusion into the first wedge summand and $i_2:\Sigma \mathbb{G}_{m}\rightarrow \Sigma\mathbb{G}_{m}\vee \Sigma\mathbb{G}_{m}\vee\Sigma(\mathbb{G}_{m}\wedge\mathbb{G}_{m}) $ the inclusion into the second wedge summand. Let $i:\Sigma(\mathbb{G}_{m}\wedge \mathbb{G}_{m})\rightarrow \Sigma\mathbb{G}_{m}\vee \Sigma\mathbb{G}_{m}\vee\Sigma(\mathbb{G}_{m}\wedge\mathbb{G}_{m}) $ be the canonical inclusion. Then in $\mathcal{H}_{\bullet}(\mathbb{Z})$ the morphism $j:=i_1\circ\Sigma\pi_1 +i_2\circ\Sigma\pi_2+i\circ\Sigma p$ is an isomorphism. Especially, the composite 
\[\begin{tikzcd}
	{\Sigma(\mathbb{G}_{m}\wedge\mathbb{G}_{m})} & {\Sigma\mathbb{G}_{m}\vee \Sigma\mathbb{G}_{m}\vee\Sigma(\mathbb{G}_{m}\wedge\mathbb{G}_{m})} & {\Sigma(\mathbb{G}_{m}\times\mathbb{G}_{m})} & {\Sigma\mathbb{G}_{m}}
	\arrow["i", hook, from=1-1, to=1-2]
	\arrow["{{{j^{-1}}}}", from=1-2, to=1-3]
	\arrow["{\Sigma\mu}", from=1-3, to=1-4]
\end{tikzcd}\]
is also equal to the Hopf construction on $\mu$.\\

\begin{lem}\label{Lemma 2.7}
	The relation $1_{1+(2)}+\eta_{1+(2)}\circ \Delta_{1+(3)}=-\epsilon_{1+(2)}$ holds.\\
\end{lem}

\begin{proof}
	We first consider the composite \[\begin{tikzcd}
		{\Sigma\mathbb{G}_{m}} & {\Sigma(\mathbb{G}_{m}\times \mathbb{G}_{m})} & {\Sigma\mathbb{G}_{m}}
		\arrow["{\Sigma\tilde{\Delta}}", from=1-1, to=1-2]
		\arrow["{\Sigma\mu}", from=1-2, to=1-3]
	\end{tikzcd}\] where $\tilde{\Delta}: \mathbb{G}_{m}\rightarrow \mathbb{G}_{m}\times \mathbb{G}_{m}$ is the corresponding diagonal morphism. Moreover this composite is just $q_{1+(1)}$. Recall that $q_{(1)}:\mathbb{G}_{m}\rightarrow \mathbb{G}_{m}; x\mapsto x^2$ is the quadratic map. On the other hand we can take the splitting and see that the composite is also equal to the following composite\[\begin{tikzcd}
	{\Sigma\mathbb{G}_{m}} && {\Sigma\mathbb{G}_{m}\vee \Sigma\mathbb{G}_{m}\vee\Sigma(\mathbb{G}_{m}\wedge\mathbb{G}_{m})} && {\Sigma\mathbb{G}_{m}} & \cdot
	\arrow["{{j\circ \Sigma\tilde{\Delta}}}", from=1-1, to=1-3]
	\arrow["{\Sigma\mu\circ j^{-1}}", from=1-3, to=1-5]
	\end{tikzcd}\]
	Now we have \begin{align*}
	j\circ \Sigma\tilde{\Delta}= (i_1\circ\Sigma\pi_1 +i_2\circ\Sigma\pi_2+i\circ\Sigma p)\circ \Sigma\tilde{\Delta}\\=
	i_1+i_2+i\circ \Delta_{1+(2)} 
		\end{align*} Thus we get \begin{align*}
		\Sigma\mu\circ j^{-1}\circ(i_1+i_2+i\circ \Delta_{1+(2)})=1_{1+(1)}+ 1_{1+(1)}+H(\mu)\circ \Delta_{1+(2)} 
		\end{align*}  
where $1_{1+(1)}$ denotes the identity on $\Sigma\mathbb{G}_{m}$. It follows that $2_{1+(1)}+H(\mu)\circ \Delta_{1+(2)}=q_{1+(1)}$. In the author's doctoral thesis \cite[3.4.3]{dong2024motivictodabrackets} the relation $q_{1+(1)}=1_{1+(1)}-\epsilon_{1+(1)}=:h_{1+(1)}$ is proved. As a consequence we get $2_{1+(1)}+H(\mu)\circ \Delta_{1+(2)}=1_{1+(1)}-\epsilon_{1+(1)}$. Therefore the relation $1_{1+(1)}+H(\mu)\circ \Delta_{1+(2)}=-\epsilon_{1+(1)}$ holds. By Lemma 2.5 we obtain the relation $1_{1+(2)}+\eta_{1+(2)}\circ \Delta_{1+(3)}=-\epsilon_{1+(2)}$.\\
\end{proof}

\section{Unstable null-Hopf relations}
\

The second motivic Hopf map is defined to be the Hopf construction on the multiplication map $\xi: \mathrm{SL}_{2}\times \mathrm{SL}_{2}\rightarrow \mathrm{SL}_{2} $. The Hopf construction $H(\xi)$ is a morphism from $\Sigma(\mathrm{SL}_{2}\wedge \mathrm{SL}_{2})\rightarrow \Sigma \mathrm{SL}_{2}$. Since $\mathrm{SL}_{2}$ can be identified as in Section 1 with $S^1\wedge\mathbb{G}_{m}^2 $, the Hopf construction $H(\xi)$ can be considered as a morphism from $S^1\wedge S^1\wedge\mathbb{G}_{m}^2\wedge S^1\wedge\mathbb{G}_{m}^2$ to $S^2\wedge\mathbb{G}_{m}^2$. Now we take the isomorphism $\lambda: S^3\wedge\mathbb{G}_{m}^4\rightarrow S^1\wedge S^1\wedge\mathbb{G}_{m}^2\wedge S^1\wedge\mathbb{G}_{m}^2$ which is given by $ t_1\wedge t_2\wedge t_3\wedge x_1\wedge x_2\wedge x_3\wedge x_4\mapsto t_1\wedge t_2\wedge  x_1\wedge x_2\wedge t_3\wedge  x_3\wedge x_4$. Then we define $\nu_{2+(2)}$ to be $H(\xi)\circ \lambda$. Thus the morphism $\nu_{2+(2)}$ is a morphism from $ S^3\wedge\mathbb{G}_{m}^4$ to $S^2\wedge\mathbb{G}_{m}^2$.\\

In the following we would like to show that $\eta_{3+(2)}\circ \nu_{3+(3)}$ is $\mathbb{A}^1$-nullhomotopic. Using the weak equivalence between $\mathrm{SL}_{2}$ and $S^1\wedge\mathbb{G}_{m}^2$ given in the proof of Proposition~\ref{Propostion 2.6} and Proposition 2.6 it is enough to show that the composite 
\[\begin{tikzcd}
	{S^2\wedge\mathrm{SL}_{2}\wedge \mathrm{SL}_{2}\wedge\mathbb{G}_{m}} && {S^2\wedge\mathrm{SL}_{2}\wedge \mathbb{G}_{m}} && {S^2\wedge\mathrm{SL}_{2}}
	\arrow["{\mathrm{id}_{S^1}\wedge H(\xi) \wedge \mathrm{id}_{\mathbb{G}_{m}}}", from=1-1, to=1-3]
	\arrow["{\mathrm{id}_{S^2}\wedge H(\alpha)}", from=1-3, to=1-5]
\end{tikzcd}\]
is nullhomotopic. The proof is not difficult, but technical; therefore we divide it into several propositions. Let $\tilde{\xi}: \mathrm{SL}_{2}\ast \mathrm{SL}_{2}\rightarrow \Sigma\mathrm{SL}_{2}$ denote the morphism induced by $\xi$ and $\tilde{\alpha}: \mathrm{SL}_{2}\ast \mathbb{G}_{m}\rightarrow \Sigma\mathrm{SL}_{2}$ the morphism induced by $\alpha$. The above composite can be written as the composite \[\begin{tikzcd}
	{S^2\wedge\mathrm{SL}_{2}\wedge \mathrm{SL}_{2}\wedge\mathbb{G}_{m}} & {S^1\wedge(\mathrm{SL}_{2}\ast \mathrm{SL}_{2})\wedge\mathbb{G}_{m}} && {S^2\wedge\mathrm{SL}_{2}\wedge \mathbb{G}_{m}} \\
	&&& {S^1\wedge(\mathrm{SL}_{2}\ast \mathbb{G}_{m})} \\
	&&& {S^2\wedge\mathrm{SL}_{2}}
	\arrow["\sim", from=1-1, to=1-2]
	\arrow["{\mathrm{id}_{S^1}\wedge \tilde\xi \wedge \mathrm{id}_{\mathbb{G}_{m}}}", from=1-2, to=1-4]
	\arrow["\sim", from=1-4, to=2-4]
	\arrow["{\mathrm{id}_{S^1}\wedge \tilde\alpha }", from=2-4, to=3-4]
\end{tikzcd}\]
Thus we only need to show that \[\begin{tikzcd}
	{S^1\wedge(\mathrm{SL}_{2}\ast \mathrm{SL}_{2})\wedge\mathbb{G}_{m}} && {S^2\wedge\mathrm{SL}_{2}\wedge \mathbb{G}_{m}} \\
	&& {S^1\wedge(\mathrm{SL}_{2}\ast \mathbb{G}_{m})} \\
	&& {S^2\wedge\mathrm{SL}_{2}}
	\arrow["{\mathrm{id}_{S^1}\wedge \tilde\xi \wedge \mathrm{id}_{\mathbb{G}_{m}}}", from=1-1, to=1-3]
	\arrow["\sim", from=1-3, to=2-3]
	\arrow["{\mathrm{id}_{S^1}\wedge \tilde\alpha }", from=2-3, to=3-3]
\end{tikzcd}\] is nullhomotopic.\\

\begin{prop} \label{Proposition 3.1}
	The diagram \[\begin{tikzcd}
		{S^1\wedge(\mathrm{SL}_{2}\ast \mathrm{SL}_{2})\wedge\mathbb{G}_{m}} & {(\mathrm{SL}_{2}\ast \mathrm{SL}_{2})\ast\mathbb{G}_{m}} & {S^1\wedge((\mathrm{SL}_{2}\ast \mathrm{SL}_{2})\times\mathbb{G}_{m}\times\mathbb{G}_{m} )} \\
		&& {S^1\wedge(\mathrm{SL}_{2}\ast \mathrm{SL}_{2})} \\
		{S^2\wedge\mathrm{SL}_{2}\wedge \mathbb{G}_{m}} \\
		{S^1\wedge(\mathrm{SL}_{2}\ast \mathbb{G}_{m})} & {S^2\wedge\mathrm{SL}_{2}}
		\arrow["{\mathrm{id}_{S^1}\wedge \tilde\xi \wedge \mathrm{id}_{\mathbb{G}_{m}}}"', from=1-1, to=3-1]
		\arrow["\sim"', from=1-2, to=1-1]
		\arrow["f", from=1-2, to=1-3]
		\arrow["g", from=1-3, to=2-3]
		\arrow["{\mathrm{id}_{S^1}\wedge \tilde\xi}", from=2-3, to=4-2]
		\arrow["\sim", from=4-1, to=3-1]
		\arrow["{\mathrm{id}_{S^1}\wedge \tilde\alpha }"', from=4-1, to=4-2]
	\end{tikzcd}\] commutes in $\mathcal{H}_{\bullet}(\mathbb{Z})$. The morphism $f$ is the composite \[\begin{tikzcd}
		{(\mathrm{SL}_{2}\ast \mathrm{SL}_{2})\ast\mathbb{G}_{m}} & {S^1\wedge((\mathrm{SL}_{2}\ast \mathrm{SL}_{2})\times\mathbb{G}_{m})} && {S^1\wedge((\mathrm{SL}_{2}\ast \mathrm{SL}_{2})\times\mathbb{G}_{m}\times\mathbb{G}_{m})}
		\arrow[from=1-1, to=1-2]
		\arrow["{\mathrm{id}_{S^1}\wedge(\mathrm{id}_{\mathrm{SL}_{2}\ast \mathrm{SL}_{2}}\times\tilde\Delta)}", from=1-2, to=1-4]
	\end{tikzcd}\] where $\tilde\Delta$ is the diagonal morphism. There is a pointed morphism $(\mathrm{SL}_{2}\ast \mathrm{SL}_{2})\times\mathbb{G}_{m}\times\mathbb{G}_{m} \rightarrow \mathrm{SL}_{2}\ast \mathrm{SL}_{2}$ given by $(\overline{(s,A,B)}, x, y)\mapsto \overline{(s,\alpha(A,x),\alpha(B,y))}$. The morphism $g$ in the diagram is induced by it.\\
\end{prop}

\begin{proof}
	From the original diagram we get \[\begin{tikzcd}
		{S^1\wedge(\mathrm{SL}_{2}\ast \mathrm{SL}_{2})\wedge\mathbb{G}_{m}} & {(\mathrm{SL}_{2}\ast \mathrm{SL}_{2})\ast\mathbb{G}_{m}} & {S^1\wedge((\mathrm{SL}_{2}\ast \mathrm{SL}_{2})\times\mathbb{G}_{m}\times\mathbb{G}_{m} )} \\
		& {(S^1\wedge \mathrm{SL}_{2})\ast\mathbb{G}_{m}} \\
		{S^2\wedge\mathrm{SL}_{2}\wedge \mathbb{G}_{m}} \\
		{S^1\wedge(\mathrm{SL}_{2}\ast \mathbb{G}_{m})} & {S^2\wedge\mathrm{SL}_{2}} & {S^1\wedge(\mathrm{SL}_{2}\ast \mathrm{SL}_{2})}
		\arrow["{{\mathrm{id}_{S^1}\wedge \tilde\xi \wedge \mathrm{id}_{\mathbb{G}_{m}}}}"', from=1-1, to=3-1]
		\arrow["\sim"', from=1-2, to=1-1]
		\arrow["f", from=1-2, to=1-3]
		\arrow["{{ \tilde\xi\ast\mathrm{id}_{\mathbb{G}_{m}}}}", from=1-2, to=2-2]
		\arrow["g", from=1-3, to=4-3]
		\arrow["\sim", from=2-2, to=3-1]
		\arrow["e", from=2-2, to=4-2]
		\arrow["\sim", from=4-1, to=3-1]
		\arrow["{{\mathrm{id}_{S^1}\wedge \tilde\alpha }}"', from=4-1, to=4-2]
		\arrow["{{\mathrm{id}_{S^1}\wedge \tilde\xi}}", from=4-3, to=4-2]
	\end{tikzcd}\]
	The morphism $e$ is induced by the pointed morphism $(S^1\wedge \mathrm{SL}_{2})\times \mathbb{G}_{m}\rightarrow S^1\wedge \mathrm{SL}_{2}; (s\wedge A, x)\mapsto s\wedge \alpha(A,x)$. By direct inspection the diagrams \[\begin{tikzcd}
		{S^1\wedge(\mathrm{SL}_{2}\ast \mathrm{SL}_{2})\wedge\mathbb{G}_{m}} & {(\mathrm{SL}_{2}\ast \mathrm{SL}_{2})\ast\mathbb{G}_{m}} \\
		\\
		{S^2\wedge\mathrm{SL}_{2}\wedge \mathbb{G}_{m}} & {(S^1\wedge \mathrm{SL}_{2})\ast\mathbb{G}_{m}}
		\arrow["{\mathrm{id}_{S^1}\wedge \tilde\xi \wedge \mathrm{id}_{\mathbb{G}_{m}}}"', from=1-1, to=3-1]
		\arrow["\sim"', from=1-2, to=1-1]
		\arrow["{\tilde\xi\ast\mathrm{id}_{\mathbb{G}_{m}}}", from=1-2, to=3-2]
		\arrow["\sim"', from=3-2, to=3-1]
	\end{tikzcd}\] and \[\begin{tikzcd}
		{(\mathrm{SL}_{2}\ast \mathrm{SL}_{2})\ast\mathbb{G}_{m}} & {S^1\wedge((\mathrm{SL}_{2}\ast \mathrm{SL}_{2})\times\mathbb{G}_{m}\times\mathbb{G}_{m} )} \\
		{(S^1\wedge \mathrm{SL}_{2})\ast\mathbb{G}_{m}} \\
		\\
		{S^2\wedge\mathrm{SL}_{2}} & {S^1\wedge(\mathrm{SL}_{2}\ast \mathrm{SL}_{2})}
		\arrow["f", from=1-1, to=1-2]
		\arrow["{{ \tilde\xi\ast\mathrm{id}_{\mathbb{G}_{m}}}}", from=1-1, to=2-1]
		\arrow["g", from=1-2, to=4-2]
		\arrow["e", from=2-1, to=4-1]
		\arrow["{{\mathrm{id}_{S^1}\wedge \tilde\xi}}", from=4-2, to=4-1]
	\end{tikzcd}\] commute.\\
	
	Thus we need to show that \[\begin{tikzcd}
		{S^2\wedge\mathrm{SL}_{2}\wedge \mathbb{G}_{m}} & {(S^1\wedge \mathrm{SL}_{2})\ast\mathbb{G}_{m}} \\
		\\
		{S^1\wedge(\mathrm{SL}_{2}\ast \mathbb{G}_{m})} & {S^2\wedge\mathrm{SL}_{2}}
		\arrow["\sim"', from=1-2, to=1-1]
		\arrow["e", from=1-2, to=3-2]
		\arrow["\sim", from=3-1, to=1-1]
		\arrow["{{{\mathrm{id}_{S^1}\wedge \tilde\alpha }}}"', from=3-1, to=3-2]
	\end{tikzcd}\] commutes. We first consider the join $(S^1\wedge \mathrm{SL}_{2})\ast\mathbb{G}_{m}$. Let $C'\mathbb{G}_{m}$ be the cone $\Delta^1\wedge \mathbb{G}_{m}$, where $\Delta^1$ is based at 0. Then there is a canonical inclusion $C'\mathbb{G}_{m}\rightarrow (S^1\wedge \mathrm{SL}_{2})\ast\mathbb{G}_{m}$ given by $s\wedge x\mapsto \overline{(s,\ast, x)} $, where $\ast$ denotes the base point of $S^1\wedge \mathrm{SL}_{2}$. In particular, we can take the cofiber $(S^1\wedge \mathrm{SL}_{2})\ast\mathbb{G}_{m}/ C'\mathbb{G}_{m}$. The canonical projection $(S^1\wedge \mathrm{SL}_{2})\ast\mathbb{G}_{m}\rightarrow  (S^1\wedge \mathrm{SL}_{2})\ast\mathbb{G}_{m}/ C'\mathbb{G}_{m}$ is a weak equivalence. Similarly, the canonical projection $(S^1\wedge \mathrm{SL}_{2})\ast\mathbb{G}_{m}/ C'\mathbb{G}_{m}\rightarrow S^2\wedge\mathrm{SL}_{2}\wedge \mathbb{G}_{m} $ is a weak equivalence, too.\\
	
	Moreover, we can define the morphism $
		\textcircled{1}: (S^1\wedge \mathrm{SL}_{2})\ast\mathbb{G}_{m}/ C'\mathbb{G}_{m}\rightarrow S^2\wedge\mathrm{SL}_{2}
$; it is induced by \begin{align*}
(S^1\wedge \mathrm{SL}_{2})\ast\mathbb{G}_{m}\rightarrow S^2\wedge\mathrm{SL}_{2}; \  \overline{(s, r\wedge A, x)}\mapsto s\wedge r\wedge \alpha(A,x) \ \ \cdot
\end{align*}
There is also a morphism $	\textcircled{2}:S^1\wedge(\mathrm{SL}_{2}\ast \mathbb{G}_{m}) \rightarrow (S^1\wedge \mathrm{SL}_{2})\ast\mathbb{G}_{m}/ C'\mathbb{G}_{m} $ given by $s\wedge \overline{(r, A,x)}\mapsto \overline{(\overline{r,s\wedge A,x})}$. Using these morphisms we obtain the following diagram \[\begin{tikzcd}
	&& {(S^1\wedge \mathrm{SL}_{2})\ast\mathbb{G}_{m}} \\
	\\
	{S^2\wedge\mathrm{SL}_{2}\wedge \mathbb{G}_{m}} && {(S^1\wedge \mathrm{SL}_{2})\ast\mathbb{G}_{m}/ C'\mathbb{G}_{m}} & {S^2\wedge\mathrm{SL}_{2}} \\
	\\
	&& {S^1\wedge(\mathrm{SL}_{2}\ast \mathbb{G}_{m})}
	\arrow["\sim"', from=1-3, to=3-1]
	\arrow["\sim"', from=1-3, to=3-3]
	\arrow["e", from=1-3, to=3-4]
	\arrow["\sim"', from=3-3, to=3-1]
	\arrow["{\textcircled{1}}", from=3-3, to=3-4]
	\arrow["\sim", from=5-3, to=3-1]
	\arrow["{\textcircled{2}}"{description}, from=5-3, to=3-3]
	\arrow["{\mathrm{id}_{S^1}\wedge \tilde\alpha}"', from=5-3, to=3-4]
\end{tikzcd}\]
It follows from the definition of $e$ and $\textcircled{1}$ that the diagram \[\begin{tikzcd}
	{(S^1\wedge \mathrm{SL}_{2})\ast\mathbb{G}_{m}} \\
	\\
	{(S^1\wedge \mathrm{SL}_{2})\ast\mathbb{G}_{m}/ C'\mathbb{G}_{m}} & {S^2\wedge\mathrm{SL}_{2}}
	\arrow["\sim"', from=1-1, to=3-1]
	\arrow["e", from=1-1, to=3-2]
	\arrow["{\textcircled{1}}", from=3-1, to=3-2]
\end{tikzcd}\] commutes. By inspection the diagram \[\begin{tikzcd}
&& {(S^1\wedge \mathrm{SL}_{2})\ast\mathbb{G}_{m}} \\
\\
{S^2\wedge\mathrm{SL}_{2}\wedge \mathbb{G}_{m}} && {(S^1\wedge \mathrm{SL}_{2})\ast\mathbb{G}_{m}/ C'\mathbb{G}_{m}}
\arrow["\sim"', from=1-3, to=3-1]
\arrow["\sim"', from=1-3, to=3-3]
\arrow["\sim"', from=3-3, to=3-1]
\end{tikzcd}\] commutes, too.\\

Now the diagram \[\begin{tikzcd}
	{(S^1\wedge \mathrm{SL}_{2})\ast\mathbb{G}_{m}/ C'\mathbb{G}_{m}} & {S^2\wedge\mathrm{SL}_{2}} \\
	\\
	{S^1\wedge(\mathrm{SL}_{2}\ast \mathbb{G}_{m})}
	\arrow["{{\textcircled{1}}}", from=1-1, to=1-2]
	\arrow["{{\textcircled{2}}}"{description}, from=3-1, to=1-1]
	\arrow["{{\mathrm{id}_{S^1}\wedge \tilde\alpha}}"', from=3-1, to=1-2]
\end{tikzcd}\] commutes by sign, as the morphism $S^2\rightarrow S^2; s\wedge r\mapsto r\wedge s$ is $-\mathrm{id}_{S^2}$. It follows from the same reason that the diagram \[\begin{tikzcd}
{S^2\wedge\mathrm{SL}_{2}\wedge \mathbb{G}_{m}} && {(S^1\wedge \mathrm{SL}_{2})\ast\mathbb{G}_{m}/ C'\mathbb{G}_{m}} \\
\\
&& {S^1\wedge(\mathrm{SL}_{2}\ast \mathbb{G}_{m})}
\arrow["\sim"', from=1-3, to=1-1]
\arrow["\sim", from=3-3, to=1-1]
\arrow["{{\textcircled{2}}}"{description}, from=3-3, to=1-3]
\end{tikzcd}\] commutes also by sign. Since these two diagrams commute by sign, it follows that the whole diagram 
\[\begin{tikzcd}
	{S^2\wedge\mathrm{SL}_{2}\wedge \mathbb{G}_{m}} & {(S^1\wedge \mathrm{SL}_{2})\ast\mathbb{G}_{m}} \\
	\\
	{S^1\wedge(\mathrm{SL}_{2}\ast \mathbb{G}_{m})} & {S^2\wedge\mathrm{SL}_{2}}
	\arrow["\sim"', from=1-2, to=1-1]
	\arrow["e", from=1-2, to=3-2]
	\arrow["\sim", from=3-1, to=1-1]
	\arrow["{{{\mathrm{id}_{S^1}\wedge \tilde\alpha }}}"', from=3-1, to=3-2]
\end{tikzcd}\] is commutative.\\
	\end{proof}

By Proposition~\ref{Proposition 3.1} it suffices to show that the composite
\[\begin{tikzcd}
	{(\mathrm{SL}_{2}\ast \mathrm{SL}_{2})\ast\mathbb{G}_{m}} & {S^1\wedge((\mathrm{SL}_{2}\ast \mathrm{SL}_{2})\times\mathbb{G}_{m}\times\mathbb{G}_{m} )} & {S^1\wedge(\mathrm{SL}_{2}\ast \mathrm{SL}_{2})}
	\arrow["f", from=1-1, to=1-2]
	\arrow["g", from=1-2, to=1-3]
\end{tikzcd}\] is $\mathbb{A}^1$-nullhomotopic. We observe that this composite fits into another commutative diagram \[\begin{tikzcd}
{(\mathrm{SL}_{2}\ast \mathrm{SL}_{2})\ast\mathbb{G}_{m}} & {S^1\wedge((\mathrm{SL}_{2}\ast \mathrm{SL}_{2})\times\mathbb{G}_{m}\times\mathbb{G}_{m} )} & {S^1\wedge(\mathrm{SL}_{2}\ast \mathrm{SL}_{2})} \\
{(S^1\wedge\mathrm{SL}_{2}\wedge \mathrm{SL}_{2})\ast\mathbb{G}_{m}} & {(S^1\wedge\mathrm{SL}_{2}\wedge \mathrm{SL}_{2})\ast(\mathbb{G}_{m}\times\mathbb{G}_{m})} & {S^2\wedge\mathrm{SL}_{2}\wedge\mathrm{SL}_{2}}
\arrow["f", from=1-1, to=1-2]
\arrow["\sim"', from=1-1, to=2-1]
\arrow["g", from=1-2, to=1-3]
\arrow["\sim", from=1-3, to=2-3]
\arrow["{\mathrm{id}\ast\tilde\Delta}", from=2-1, to=2-2]
\arrow["{\hat{g}}", from=2-2, to=2-3]
\end{tikzcd}\] where the morphism $\hat{g}$ is induced by $(S^1\wedge\mathrm{SL}_{2}\wedge \mathrm{SL}_{2})\times(\mathbb{G}_{m}\times\mathbb{G}_{m})\rightarrow S^1\wedge\mathrm{SL}_{2}\wedge \mathrm{SL}_{2};\  (t\wedge A\wedge B,(x,y))\mapsto t\wedge \alpha(A,x)\wedge \alpha(B,y)$.\\

\begin{prop} \label{Proposition 3.2}
	The composite $\hat{g}\circ (\mathrm{id}\ast\tilde{\Delta})$ is $\mathbb{A}^1$-nullhomotopic.\\
\end{prop}
\begin{proof}
We first note that $(S^1\wedge\mathrm{SL}_{2}\wedge \mathrm{SL}_{2})\ast\mathbb{G}_{m}$ is canonically weakly equivalent to $S^2\wedge \mathrm{SL}_{2}\wedge \mathrm{SL}_{2}\wedge\mathbb{G}_{m} $. We denote the weak equivalence by $p$. Therefore we form the composite \[\begin{tikzcd}
	{(S^1\wedge\mathrm{SL}_{2}\wedge \mathrm{SL}_{2})\ast\mathbb{G}_{m}} & {(S^1\wedge\mathrm{SL}_{2}\wedge \mathrm{SL}_{2})\ast(\mathbb{G}_{m}\times\mathbb{G}_{m})} & {S^2\wedge\mathrm{SL}_{2}\wedge\mathrm{SL}_{2}} \\
	{S^2\wedge \mathrm{SL}_{2}\wedge \mathrm{SL}_{2}\wedge\mathbb{G}_{m}}
	\arrow["{{{\mathrm{id}\ast\tilde\Delta}}}", from=1-1, to=1-2]
	\arrow["{{{\hat{g}}}}", from=1-2, to=1-3]
	\arrow["{p^{-1}}", from=2-1, to=1-1]
\end{tikzcd}\] Moreover, we observe that $(S^1\wedge\mathrm{SL}_{2}\wedge \mathrm{SL}_{2})\ast(\mathbb{G}_{m}\times\mathbb{G}_{m})$ is canonically weakly equivalent to $S^2\wedge\mathrm{SL}_{2}\wedge \mathrm{SL}_{2}\wedge (\mathbb{G}_{m}\times \mathbb{G}_{m})$. We denote this weak equivalence by $\tilde p$. Via the stable splitting for $S^1\wedge(\mathbb{G}_{m}\times \mathbb{G}_{m}) $ (see p. 12) we get the following weak equivalence \[\begin{tikzcd}
{S^2\wedge(\mathrm{SL}_{2})^2\wedge (\mathbb{G}_{m}\times \mathbb{G}_{m})} \\
{(S^2\wedge(\mathrm{SL}_{2})^2\wedge \mathbb{G}_{m})\vee (S^2\wedge(\mathrm{SL}_{2})^2\wedge \mathbb{G}_{m})\vee(S^2\wedge(\mathrm{SL}_{2})^2 \wedge\mathbb{G}_{m}\wedge \mathbb{G}_{m} )}
\arrow["{\hat{j}}", from=1-1, to=2-1]
\end{tikzcd}\] where $(\mathrm{SL}_{2})^2$ is just $\mathrm{SL}_{2}\wedge \mathrm{SL}_{2}$. In the following we denote this wedge sum by $\mathcal{W}$.\\

In particular, we have the following commutative diagram \[\begin{tikzcd}
	{(S^1\wedge\mathrm{SL}_{2}\wedge \mathrm{SL}_{2})\ast\mathbb{G}_{m}} & {(S^1\wedge\mathrm{SL}_{2}\wedge \mathrm{SL}_{2})\ast(\mathbb{G}_{m}\times\mathbb{G}_{m})} & {S^2\wedge\mathrm{SL}_{2}\wedge\mathrm{SL}_{2}} \\
	& {S^2\wedge\mathrm{SL}_{2}\wedge \mathrm{SL}_{2}\wedge(\mathbb{G}_{m}\times\mathbb{G}_{m})} \\
	\\
	{S^2\wedge \mathrm{SL}_{2}\wedge \mathrm{SL}_{2}\wedge\mathbb{G}_{m}} & {\mathcal{W}}
	\arrow["{{{{\mathrm{id}\ast\tilde\Delta}}}}", from=1-1, to=1-2]
	\arrow["{{{{\hat{g}}}}}", from=1-2, to=1-3]
	\arrow["{\tilde{p}^{-1}}", from=2-2, to=1-2]
	\arrow["{\hat{j}}"', from=2-2, to=4-2]
	\arrow["{p^{-1}}", from=4-1, to=1-1]
	\arrow["{\mathrm{id}_{S^2\wedge \mathrm{SL}_{2}\wedge \mathrm{SL}_{2}\wedge\tilde\Delta}}"{description}, from=4-1, to=2-2]
	\arrow["{\hat{g}\circ \tilde{p}^{-1}\circ \hat{j}^{-1}}"', from=4-2, to=1-3]
\end{tikzcd}\]
Thus it suffices to show that the composite $\hat{g}\circ \tilde{p}^{-1}\circ \hat{j}^{-1}\circ\hat{j}\circ \mathrm{id}_{S^2\wedge \mathrm{SL}_{2}\wedge \mathrm{SL}_{2}}\wedge\tilde\Delta$ is $0$ in $\mathcal{H}_{\bullet}(\mathbb{Z})$.\\

In the next step we would like to write this composite as a sum of three  morphisms. For this we look at the composite \[\begin{tikzcd}
	{S^2\wedge\mathrm{SL}_{2}\wedge \mathrm{SL}_{2}\wedge(\mathbb{G}_{m}\times\mathbb{G}_{m})} & {\mathcal{W}} && {S^2\wedge\mathrm{SL}_{2}\wedge\mathrm{SL}_{2}}
	\arrow["{\hat j }"', from=1-1, to=1-2]
	\arrow["{\hat{g}\circ \tilde{p}^{-1}\circ \hat{j}^{-1}}"', from=1-2, to=1-4]
\end{tikzcd}\] Let $\pi_1$ be the projection $S^2\wedge\mathrm{SL}_{2}\wedge \mathrm{SL}_{2}\wedge(\mathbb{G}_{m}\times\mathbb{G}_{m}) \rightarrow S^2\wedge\mathrm{SL}_{2}\wedge \mathrm{SL}_{2}\wedge \mathbb{G}_{m}$ which is induced by the projection $\mathbb{G}_{m}\times\mathbb{G}_{m}\rightarrow \mathbb{G}_{m}$ onto the first factor. Let $\pi_2$ be the projection $S^2\wedge\mathrm{SL}_{2}\wedge \mathrm{SL}_{2}\wedge(\mathbb{G}_{m}\times\mathbb{G}_{m}) \rightarrow S^2\wedge\mathrm{SL}_{2}\wedge \mathrm{SL}_{2}\wedge \mathbb{G}_{m}$ which is induced by the projection $\mathbb{G}_{m}\times\mathbb{G}_{m}\rightarrow \mathbb{G}_{m}$ onto the second factor. Let $\pi_3$ be the canonical projection $S^2\wedge\mathrm{SL}_{2}\wedge \mathrm{SL}_{2}\wedge(\mathbb{G}_{m}\times\mathbb{G}_{m}) \rightarrow S^2\wedge\mathrm{SL}_{2}\wedge \mathrm{SL}_{2}\wedge \mathbb{G}_{m}\wedge \mathbb{G}_{m}$. Let $i_1: S^2\wedge\mathrm{SL}_{2}\wedge \mathrm{SL}_{2}\wedge \mathbb{G}_{m}\rightarrow \mathcal{W} $ be the inclusion into the first wedge summand and $i_2: S^2\wedge\mathrm{SL}_{2}\wedge \mathrm{SL}_{2}\wedge \mathbb{G}_{m}\rightarrow \mathcal{W} $ the inclusion into the second summand. Let $i_3: S^2\wedge\mathrm{SL}_{2}\wedge \mathrm{SL}_{2}\wedge \mathbb{G}_{m}\wedge\mathbb{G}_{m} \rightarrow \mathcal{W} $ be the inclusion into the third summand.\\

Then the composite \[\begin{tikzcd}
	{S^2\wedge\mathrm{SL}_{2}\wedge \mathrm{SL}_{2}\wedge(\mathbb{G}_{m}\times\mathbb{G}_{m})} & {\mathcal{W}} && {S^2\wedge\mathrm{SL}_{2}\wedge\mathrm{SL}_{2}}
	\arrow["{\hat j }"', from=1-1, to=1-2]
	\arrow["{\hat{g}\circ \tilde{p}^{-1}\circ \hat{j}^{-1}}"', from=1-2, to=1-4]
\end{tikzcd}\] can be written as $\hat{g}\circ \tilde{p}^{-1}\circ \hat{j}^{-1}\circ i_1\circ \pi_1+ \hat{g}\circ \tilde{p}^{-1}\circ \hat{j}^{-1}\circ i_2\circ \pi_2+ \hat{g}\circ \tilde{p}^{-1}\circ \hat{j}^{-1}\circ i_3\circ \pi_3$. Hence we would like to analyze the following three morphisms in the next two propositions, namely \begin{align*}
	\hat{g}\circ \tilde{p}^{-1}\circ \hat{j}^{-1}\circ i_1\circ \pi_1\circ \mathrm{id}_{S^2\wedge \mathrm{SL}_{2}\wedge \mathrm{SL}_{2}}\wedge\tilde\Delta, \\
	\hat{g}\circ \tilde{p}^{-1}\circ \hat{j}^{-1}\circ i_2\circ \pi_2 \circ\mathrm{id}_{S^2\wedge \mathrm{SL}_{2}\wedge \mathrm{SL}_{2}}\wedge\tilde\Delta, \\
	\hat{g}\circ \tilde{p}^{-1}\circ \hat{j}^{-1}\circ i_3\circ \pi_3\circ \mathrm{id}_{S^2\wedge \mathrm{SL}_{2}\wedge \mathrm{SL}_{2}}\wedge\tilde\Delta
\end{align*}
In the proof of Proposition~\ref{Propostion 2.6} we give an explicit weak equivalence between $\mathrm{SL}_2$ and $S^1\wedge \mathbb{G}_{m}^2$. Via this weak equivalence $S^2\wedge\mathrm{SL}_{2}\wedge \mathrm{SL}_{2}\wedge \mathbb{G}_{m}$ can be identified with $S^2\wedge S^1\wedge \mathbb{G}_{m}^2\wedge S^1\wedge \mathbb{G}_{m}^2\wedge \mathbb{G}_{m}$. Throughout the paper we identify $S^4\wedge \mathbb{G}_{m}^5 $ with $S^2\wedge S^1\wedge \mathbb{G}_{m}^2\wedge S^1\wedge \mathbb{G}_{m}^2\wedge \mathbb{G}_{m}$ via the isomorphism $(\ast)$ \begin{align*}
	t_1\wedge t_2\wedge t_3\wedge t_4\wedge x_1\wedge x_2\wedge x_3\wedge x_4\wedge x_5\mapsto 	t_1\wedge t_2\wedge t_3\wedge  x_1\wedge x_2 \wedge t_4\wedge x_3\wedge x_4\wedge x_5
\end{align*}
Similarly, the motivic space $S^2\wedge\mathrm{SL}_{2}\wedge \mathrm{SL}_{2}$ can be identified with $S^2\wedge S^1\wedge \mathbb{G}_{m}^2\wedge S^1\wedge \mathbb{G}_{m}^2$. We identify $S^4\wedge \mathbb{G}_{m}^4$ with $S^2\wedge S^1\wedge \mathbb{G}_{m}^2\wedge S^1\wedge \mathbb{G}_{m}^2$ via the isomorphism $(\ast \ast)$ \begin{align*}
	t_1\wedge t_2\wedge t_3\wedge t_4\wedge x_1\wedge x_2\wedge x_3\wedge x_4\mapsto
	t_1\wedge t_2\wedge t_3\wedge  x_1\wedge x_2\wedge t_4\wedge x_3\wedge x_4 \ \ \cdot 
\end{align*}
\end{proof}
\

\begin{prop}\label{Proposition 3.3}
Via the isomorphisms above the morphisms $\hat{g}\circ \tilde{p}^{-1}\circ \hat{j}^{-1}\circ i_1\circ \pi_1\circ \mathrm{id}_{S^2\wedge \mathrm{SL}_{2}\wedge \mathrm{SL}_{2}}\wedge\tilde\Delta$ and $\hat{g}\circ \tilde{p}^{-1}\circ \hat{j}^{-1}\circ i_2\circ \pi_2\circ \mathrm{id}_{S^2\wedge \mathrm{SL}_{2}\wedge \mathrm{SL}_{2}}\wedge\tilde\Delta$ can be identified with $\eta_{4+(4)}$.\\
\end{prop}

\begin{proof}
We first start with $\hat{g}\circ \tilde{p}^{-1}\circ \hat{j}^{-1}\circ i_2\circ \pi_2\circ \mathrm{id}_{S^2\wedge \mathrm{SL}_{2}\wedge \mathrm{SL}_{2}}\wedge\tilde\Delta$. Let $\tilde i_2:S^2\wedge\mathrm{SL}_{2}\wedge \mathrm{SL}_{2}\wedge \mathbb{G}_{m}\rightarrow S^2\wedge\mathrm{SL}_{2}\wedge \mathrm{SL}_{2}\wedge(\mathbb{G}_{m}\times\mathbb{G}_{m}) $ be the morphism which is induced by the projection $\mathbb{G}_{m}\times\mathbb{G}_{m}\rightarrow \mathbb{G}_{m}; (x,y)\mapsto y$. Then we get $\hat{g}\circ \tilde{p}^{-1}\circ \hat{j}^{-1}\circ i_2\circ \pi_2\circ \mathrm{id}_{S^2\wedge \mathrm{SL}_{2}\wedge \mathrm{SL}_{2}}\wedge\tilde\Delta=\hat{g}\circ \tilde{p}^{-1}\circ \hat{j}^{-1}\circ i_2= \hat{g}\circ \tilde{p}^{-1}\circ \tilde i_2$. Let $\hat{i}_2: (S^1\wedge\mathrm{SL}_{2}\wedge \mathrm{SL}_{2})\ast\mathbb{G}_{m}\rightarrow (S^1\wedge\mathrm{SL}_{2}\wedge \mathrm{SL}_{2})\ast(\mathbb{G}_{m}\times \mathbb{G}_{m})$ be the corresponding morphism induced by the projection $\mathbb{G}_{m}\times\mathbb{G}_{m}\rightarrow \mathbb{G}_{m}; (x,y)\mapsto y$. Then we get the following commutative diagram
\[\begin{tikzcd}
	{(S^1\wedge\mathrm{SL}_{2}\wedge \mathrm{SL}_{2})\ast\mathbb{G}_{m}} & {(S^1\wedge\mathrm{SL}_{2}\wedge \mathrm{SL}_{2})\ast(\mathbb{G}_{m}\times \mathbb{G}_{m})} & {S^2\wedge\mathrm{SL}_{2}\wedge \mathrm{SL}_{2}} \\
	{S^2\wedge\mathrm{SL}_{2}\wedge \mathrm{SL}_{2}\wedge \mathbb{G}_{m}} & {S^2\wedge\mathrm{SL}_{2}\wedge \mathrm{SL}_{2}\wedge(\mathbb{G}_{m}\times\mathbb{G}_{m})}
	\arrow["{\hat{i}_2}", from=1-1, to=1-2]
	\arrow["{\hat{g}}", from=1-2, to=1-3]
	\arrow["{p^{-1}}", from=2-1, to=1-1]
	\arrow["{\tilde i_2}"', from=2-1, to=2-2]
	\arrow["{\tilde{p}^{-1}}"', from=2-2, to=1-2]
\end{tikzcd}\] It follows that $\hat{g}\circ \tilde{p}^{-1}\circ \tilde i_2=\hat g\circ \hat i_2\circ p^{-1}$.\\

Next we consider the motivic space $(S^1\wedge \mathrm{SL}_{2})\wedge (\mathrm{SL}_{2}\ast \mathbb{G}_{m})$. There is a natural weak equivalence $q: (S^1\wedge \mathrm{SL}_{2})\wedge (\mathrm{SL}_{2}\ast \mathbb{G}_{m})\rightarrow S^2\wedge\mathrm{SL}_{2}\wedge \mathrm{SL}_{2}\wedge \mathbb{G}_{m}$ which is the composite \[\begin{tikzcd}
	{(S^1\wedge \mathrm{SL}_{2})\wedge (\mathrm{SL}_{2}\ast \mathbb{G}_{m})} & {S^1\wedge \mathrm{SL}_{2}\wedge S^1\wedge\mathrm{SL}_{2}\wedge \mathbb{G}_{m}} \\
	& {S^2\wedge\mathrm{SL}_{2}\wedge \mathrm{SL}_{2}\wedge \mathbb{G}_{m}}
	\arrow["\sim", from=1-1, to=1-2]
	\arrow["\sim", from=1-2, to=2-2]
\end{tikzcd}\] where the second weak equivalence is given by $t\wedge A\wedge s \wedge B\wedge x\mapsto s\wedge t\wedge A\wedge B\wedge x$. Moreover, we can also define a morphism $\bar g: (S^1\wedge \mathrm{SL}_{2})\wedge (\mathrm{SL}_{2}\ast \mathbb{G}_{m})\rightarrow S^2\wedge\mathrm{SL}_{2}\wedge \mathrm{SL}_{2}; (t\wedge A)\wedge \overline{(s,B,x)}\mapsto s\wedge t\wedge A \wedge \alpha(B,x)$. Let $C'\mathbb{G}_{m}$ be the cone $\Delta^1\wedge \mathbb{G}_{m}$ for $\Delta^1$ based at 0. Then there is a canonical inclusion $C'\mathbb{G}_{m}\rightarrow (S^1\wedge\mathrm{SL}_{2}\wedge \mathrm{SL}_{2})\ast\mathbb{G}_{m}$. In particular, we consider now the cofiber $(S^1\wedge\mathrm{SL}_{2}\wedge \mathrm{SL}_{2})\ast\mathbb{G}_{m}/C'\mathbb{G}_{m} $. We can also define a morphism $\theta_1: (S^1\wedge \mathrm{SL}_{2})\wedge (\mathrm{SL}_{2}\ast \mathbb{G}_{m})\rightarrow (S^1\wedge\mathrm{SL}_{2}\wedge \mathrm{SL}_{2})\ast\mathbb{G}_{m}/C'\mathbb{G}_{m}; (t\wedge A)\wedge \overline{(s,B,x)}\mapsto \overline{(\overline {s, t\wedge A\wedge B, x})}$. Let $\theta_2:(S^1\wedge\mathrm{SL}_{2}\wedge \mathrm{SL}_{2})\ast\mathbb{G}_{m}\rightarrow (S^1\wedge\mathrm{SL}_{2}\wedge \mathrm{SL}_{2})\ast\mathbb{G}_{m}/C'\mathbb{G}_{m}$ be the canonical projection. Altogether we get a commutative diagram \[\begin{tikzcd}
	& {(S^1\wedge\mathrm{SL}_{2}\wedge \mathrm{SL}_{2})\ast\mathbb{G}_{m}} \\
	{S^2\wedge\mathrm{SL}_{2}\wedge \mathrm{SL}_{2}\wedge \mathbb{G}_{m}} & {(S^1\wedge\mathrm{SL}_{2}\wedge \mathrm{SL}_{2})\ast\mathbb{G}_{m}/C'\mathbb{G}_{m} } & {S^2\wedge\mathrm{SL}_{2}\wedge \mathrm{SL}_{2}} \\
	& {(S^1\wedge \mathrm{SL}_{2})\wedge (\mathrm{SL}_{2}\ast \mathbb{G}_{m})} & \cdot
	\arrow["p"', from=1-2, to=2-1]
	\arrow["{\theta_2}", from=1-2, to=2-2]
	\arrow["{\hat g\circ \hat i_2}", from=1-2, to=2-3]
	\arrow[from=2-2, to=2-1]
	\arrow[from=2-2, to=2-3]
	\arrow["q", from=3-2, to=2-1]
	\arrow["{\theta_1}"', from=3-2, to=2-2]
	\arrow["{\bar g}"', from=3-2, to=2-3]
\end{tikzcd}\]
In the above diagram the morphism $(S^1\wedge\mathrm{SL}_{2}\wedge \mathrm{SL}_{2})\ast\mathbb{G}_{m}/C'\mathbb{G}_{m} \rightarrow S^2\wedge\mathrm{SL}_{2}\wedge \mathrm{SL}_{2}\wedge \mathbb{G}_{m}$ is induced by $p$. Similarly, the morphism $(S^1\wedge\mathrm{SL}_{2}\wedge \mathrm{SL}_{2})\ast\mathbb{G}_{m}/C'\mathbb{G}_{m}\rightarrow S^2\wedge\mathrm{SL}_{2}\wedge \mathrm{SL}_{2}$ is induced by $\hat g\circ \hat i_2$. From the commutativity of the diagram we conclude that $\hat g\circ \hat i_2\circ p^{-1}= \bar g\circ q^{-1}$. So we look at $\bar g\circ q^{-1}$.\\

By unraveling the definitions of the morphisms we see that $\bar g\circ q^{-1}$ is equal to the composite\[\begin{tikzcd}
	{S^2\wedge\mathrm{SL}_{2}\wedge\mathrm{SL}_{2}\wedge \mathbb{G}_{m}} & {S^1\wedge\mathrm{SL}_{2}\wedge S^1\wedge \mathrm{SL}_{2}\wedge \mathbb{G}_{m}} \\
	& {S^1\wedge\mathrm{SL}_{2}\wedge S^1\wedge \mathrm{SL}_{2}} \\
	& {S^2\wedge\mathrm{SL}_{2}\wedge\mathrm{SL}_{2}}
	\arrow["\sim", from=1-1, to=1-2]
	\arrow["{\mathrm{id}_{S^1\wedge\mathrm{SL}_{2}}\wedge H(\alpha)}", from=1-2, to=2-2]
	\arrow["\sim", from=2-2, to=3-2]
\end{tikzcd}\] where the first isomorphism is given by $s\wedge t\wedge A\wedge B\wedge x\mapsto t\wedge A\wedge s \wedge B\wedge x$. The second isomorphism is defined by $t\wedge A\wedge s\wedge B\mapsto s\wedge t\wedge A\wedge B$.

Via the identifications $(\ast)$ and $(\ast \ast)$ given on page 20 this composite can be written as the following composite \[\begin{tikzcd}
	{S^4\wedge\mathbb{G}_{m}^5} & {S^4\wedge\mathbb{G}_{m}^5} \\
	& {S^4\wedge\mathbb{G}_{m}^4} \\
	& {S^4\wedge\mathbb{G}_{m}^4}
	\arrow["{\textcircled{1}}", from=1-1, to=1-2]
	\arrow["{\mathrm{id}_{S^2}\wedge H(\alpha)\wedge\mathrm{id}_{\mathbb{G}_m^2}}", from=1-2, to=2-2]
	\arrow["{\textcircled{2}}", from=2-2, to=3-2]
\end{tikzcd}\] where the isomorphism $\textcircled{1}$ is given by $	t_1\wedge t_2\wedge t_3\wedge t_4\wedge x_1\wedge x_2\wedge x_3\wedge x_4\wedge x_5\mapsto t_2\wedge t_3\wedge t_1\wedge t_4\wedge x_3\wedge x_4\wedge x_5\wedge x_1\wedge x_2 $. The isomorphism $\textcircled{2}$ is defined by $t_1\wedge t_2\wedge t_3\wedge t_4\wedge x_1\wedge x_2\wedge x_3\wedge x_4\mapsto t_3\wedge t_1\wedge t_2\wedge t_4\wedge x_3\wedge x_4\wedge x_1\wedge x_2$.\\

We note that the permutation $S^4\rightarrow S^4; t_1\wedge t_2\wedge t_3\wedge t_4\mapsto t_2\wedge t_3\wedge t_1\wedge t_4$ is $\mathrm{id}_{S^3}$. Therefore the morphism $\textcircled{1}$ is equal to the morphism given by $t_1\wedge t_2\wedge t_3\wedge t_4\wedge x_1\wedge x_2\wedge x_3\wedge x_4\wedge x_5\mapsto t_1\wedge t_2\wedge t_3\wedge t_4\wedge x_3\wedge x_4\wedge x_5\wedge x_1\wedge x_2 $. By Corollary~\ref{Corollary 2.3} this morphism equals to the identity. Via the same arguments we can show that the morphism $\textcircled{2}$ is also the identity. By Proposition~\ref{Propostion 2.6} the morphism $\mathrm{id}_{S^2}\wedge H(\alpha)\wedge\mathrm{id}_{\mathbb{G}_m^2}$ is equal to $\eta_{4+(4)}$. Using similar arguments we can further show that $\hat{g}\circ \tilde{p}^{-1}\circ \hat{j}^{-1}\circ i_1\circ \pi_1\circ \mathrm{id}_{S^2\wedge \mathrm{SL}_{2}\wedge \mathrm{SL}_{2}}\wedge\tilde\Delta$ can be identified with $\eta_{4+(4)}$, too.\\
\end{proof}

\begin{prop}\label{Proposition 3.4}
	Via the identifications $(\ast)$ and $(\ast \ast)$ on page 20 the morphism $\hat{g}\circ \tilde{p}^{-1}\circ \hat{j}^{-1}\circ i_3\circ \pi_3\circ \mathrm{id}_{S^2\wedge \mathrm{SL}_{2}\wedge \mathrm{SL}_{2}}\wedge\tilde\Delta$ is equal to $\eta_{4+(4)}\circ \eta_{4+(5)}\circ \Delta_{4+(6)}$.\\
	\end{prop}

\begin{proof}
Recall that $\Delta$ denotes the diagonal morphism $\mathbb{G}_{m}\rightarrow \mathbb{G}_{m}\wedge \mathbb{G}_{m}$. The morphism $\hat{g}\circ \tilde{p}^{-1}\circ \hat{j}^{-1}\circ i_3\circ \pi_3\circ \mathrm{id}_{S^2\wedge \mathrm{SL}_{2}\wedge \mathrm{SL}_{2}}\wedge\tilde\Delta$ is then equal to $\hat{g}\circ \tilde{p}^{-1}\circ \hat{j}^{-1}\circ i_3 \circ \mathrm{id}_{S^2\wedge \mathrm{SL}_{2}\wedge \mathrm{SL}_{2}}\wedge \Delta$. We consider now the motivic space $(\mathrm{SL}_{2}\ast \mathbb{G}_{m}) \wedge (\mathrm{SL}_{2}\ast \mathbb{G}_{m}) $. We define a morphism $\gamma:(\mathrm{SL}_{2}\ast \mathbb{G}_{m}) \wedge (\mathrm{SL}_{2}\ast \mathbb{G}_{m})\rightarrow S^2\wedge\mathrm{SL}_{2}\wedge \mathrm{SL}_{2}\wedge (\mathbb{G}_{m}\times \mathbb{G}_{m}); \overline{(s,A,x)}\wedge \overline{(t,B,y)}\mapsto s\wedge t\wedge A\wedge B\wedge (x,y) $. Moreover, there is also a natural weak equivalence $\gamma':(\mathrm{SL}_{2}\ast \mathbb{G}_{m}) \wedge (\mathrm{SL}_{2}\ast \mathbb{G}_{m})\rightarrow S^2\wedge\mathrm{SL}_{2}\wedge \mathrm{SL}_{2}\wedge \mathbb{G}_{m}\wedge \mathbb{G}_{m}; \overline{(s,A,x)}\wedge \overline{(t,B,y)}\mapsto s\wedge t\wedge A\wedge B\wedge x\wedge y$. We would like to show that the diagram \[\begin{tikzcd}
	{(\mathrm{SL}_{2}\ast \mathbb{G}_{m}) \wedge (\mathrm{SL}_{2}\ast \mathbb{G}_{m})} & {S^2\wedge\mathrm{SL}_{2}\wedge \mathrm{SL}_{2}\wedge (\mathbb{G}_{m}\times \mathbb{G}_{m})} \\
	{S^2\wedge\mathrm{SL}_{2}\wedge \mathrm{SL}_{2}\wedge \mathbb{G}_{m}\wedge \mathbb{G}_{m}} & {\mathcal{W}}
	\arrow["\gamma", from=1-1, to=1-2]
	\arrow["{\gamma'}"', from=1-1, to=2-1]
	\arrow["{\hat j}", from=1-2, to=2-2]
	\arrow["{i_3}", from=2-1, to=2-2]
\end{tikzcd}\]
commutes.\\

The smash product $(\mathrm{SL}_{2}\ast \mathbb{G}_{m}) \wedge (\mathrm{SL}_{2}\ast \mathbb{G}_{m})$ inherits a coproduct structure from $\mathrm{SL}_{2}\ast \mathbb{G}_{m}\simeq S^1\wedge \mathrm{SL}_{2}\wedge \mathbb{G}_{m} $; so there is a morphism \[\begin{tikzcd}
	{(\mathrm{SL}_{2}\ast \mathbb{G}_{m}) \wedge (\mathrm{SL}_{2}\ast \mathbb{G}_{m})} & {((\mathrm{SL}_{2}\ast \mathbb{G}_{m}) \wedge (\mathrm{SL}_{2}\ast \mathbb{G}_{m}))\vee ((\mathrm{SL}_{2}\ast \mathbb{G}_{m}) \wedge (\mathrm{SL}_{2}\ast \mathbb{G}_{m}))}
	\arrow[from=1-1, to=1-2]
\end{tikzcd}\] which is a part of the coproduct structure on $(\mathrm{SL}_{2}\ast \mathbb{G}_{m}) \wedge (\mathrm{SL}_{2}\ast \mathbb{G}_{m})$. Furthermore there is also such a morphism for $S^2\wedge\mathrm{SL}_{2}\wedge \mathrm{SL}_{2}\wedge (\mathbb{G}_{m}\times \mathbb{G}_{m})$, since this motivic space is a coproduct, too. Now by definition $\gamma$ is compatible with these two morphisms which are parts of the corresponding coproduct structures. It follows that the composite $\hat j\circ \gamma$ can be written as the sum $i_1\circ \pi_1\circ \gamma+ i_2\circ \pi_2\circ \gamma+ i_3\circ \pi_3\circ \gamma $.\\

The morphism $ \pi_1\circ \gamma$ is given by $\overline{(s,A,x)}\wedge \overline{(t,B,y)}\mapsto s\wedge t\wedge A\wedge B\wedge x$. By \cite[Lemma C.2]{motihopf} the morphism $\mathrm{SL}_{2}\ast \mathbb{G}_{m}\rightarrow S^1\wedge\mathrm{SL}_{2}; \overline{(t,B,y)}\mapsto t\wedge B $ is nullhomotopic. Thus the summand $i_1\circ \pi_1\circ \gamma$ is nullhomotopic. Similarly, the summand $i_2\circ \pi_2\circ \gamma$ is also nullhomotopic. Hence the composite  $\hat j\circ \gamma$ is equal to  $i_3\circ \pi_3\circ \gamma $. We see immediately that $i_3\circ \pi_3\circ \gamma $ equals to $i_3\circ \gamma'$. We conclude that $\hat{g}\circ \tilde{p}^{-1}\circ \hat{j}^{-1}\circ i_3 \circ \mathrm{id}_{S^2\wedge \mathrm{SL}_{2}\wedge \mathrm{SL}_{2}}\wedge \Delta$ is equal to $\hat{g}\circ \tilde{p}^{-1}\circ \gamma \circ \gamma'^{-1} \circ \mathrm{id}_{S^2\wedge \mathrm{SL}_{2}\wedge \mathrm{SL}_{2}}\wedge \Delta$.\\

In the next step we define a morphism $g': (\mathrm{SL}_{2}\ast \mathbb{G}_{m}) \wedge (\mathrm{SL}_{2}\ast \mathbb{G}_{m})\rightarrow S^2\wedge\mathrm{SL}_{2}\wedge \mathrm{SL}_{2}$ by sending $\overline{(s,A,x)}\wedge \overline{(t,B,y)}$ to $s\wedge t\wedge \alpha(A,x)\wedge \alpha(B,y)$. We want to show that the diagram \[\begin{tikzcd}
	& {S^2\wedge\mathrm{SL}_{2}\wedge \mathrm{SL}_{2}} \\
	\\
	& {(S^1\wedge\mathrm{SL}_{2}\wedge \mathrm{SL}_{2})\ast (\mathbb{G}_{m}\times \mathbb{G}_{m})} \\
	{(\mathrm{SL}_{2}\ast \mathbb{G}_{m}) \wedge (\mathrm{SL}_{2}\ast \mathbb{G}_{m})} & {S^2\wedge\mathrm{SL}_{2}\wedge \mathrm{SL}_{2}\wedge (\mathbb{G}_{m}\times \mathbb{G}_{m})}
	\arrow["{\hat g}"', from=3-2, to=1-2]
	\arrow["{\tilde p}", from=3-2, to=4-2]
	\arrow["{g'}"{description}, from=4-1, to=1-2]
	\arrow["\gamma", from=4-1, to=4-2]
\end{tikzcd}\] commutes in $\mathcal{H}_{\bullet}(\mathbb{Z})$. Let $C'(\mathbb{G}_{m}\times \mathbb{G}_{m})$ be the cone $\Delta^1\wedge (\mathbb{G}_{m}\times \mathbb{G}_{m})$ for $\Delta^1$ based at 0. There is a canonical inclusion $C'(\mathbb{G}_{m}\times \mathbb{G}_{m})\rightarrow (S^1\wedge\mathrm{SL}_{2}\wedge \mathrm{SL}_{2})\ast (\mathbb{G}_{m}\times \mathbb{G}_{m})$; thus we can take the cofiber $(S^1\wedge\mathrm{SL}_{2}\wedge \mathrm{SL}_{2})\ast (\mathbb{G}_{m}\times \mathbb{G}_{m})/ C'(\mathbb{G}_{m}\times \mathbb{G}_{m})$. The projection $(S^1\wedge\mathrm{SL}_{2}\wedge \mathrm{SL}_{2})\ast (\mathbb{G}_{m}\times \mathbb{G}_{m})\rightarrow (S^1\wedge\mathrm{SL}_{2}\wedge \mathrm{SL}_{2})\ast (\mathbb{G}_{m}\times \mathbb{G}_{m})/ C'(\mathbb{G}_{m}\times \mathbb{G}_{m}) $ is a weak equivalence. The morphism $\tilde p$ factors as follows \[\begin{tikzcd}
{(S^1\wedge\mathrm{SL}_{2}\wedge \mathrm{SL}_{2})\ast (\mathbb{G}_{m}\times \mathbb{G}_{m})} \\
{(S^1\wedge\mathrm{SL}_{2}\wedge \mathrm{SL}_{2})\ast (\mathbb{G}_{m}\times \mathbb{G}_{m})/ C'(\mathbb{G}_{m}\times \mathbb{G}_{m})} \\
{S^2\wedge\mathrm{SL}_{2}\wedge \mathrm{SL}_{2}\wedge (\mathbb{G}_{m}\times \mathbb{G}_{m})} & \cdot
\arrow["\sim", from=1-1, to=2-1]
\arrow[from=2-1, to=3-1]
\end{tikzcd}\]
Furthermore there is also a morphism $\delta: (\mathrm{SL}_{2}\ast \mathbb{G}_{m}) \wedge (\mathrm{SL}_{2}\ast \mathbb{G}_{m})\rightarrow (S^1\wedge\mathrm{SL}_{2}\wedge \mathrm{SL}_{2})\ast (\mathbb{G}_{m}\times \mathbb{G}_{m})/ C'(\mathbb{G}_{m}\times \mathbb{G}_{m})$
given by $\overline{(s,A,x)}\wedge \overline{(t,B,y)}\mapsto \overline{(\overline{s,t\wedge A\wedge B, (x,y) })}$. We consider the following diagram \[\begin{tikzcd}
	& {S^2\wedge\mathrm{SL}_{2}\wedge \mathrm{SL}_{2}} \\
	& {(S^1\wedge\mathrm{SL}_{2}\wedge \mathrm{SL}_{2})\ast (\mathbb{G}_{m}\times \mathbb{G}_{m})} \\
	& {(S^1\wedge\mathrm{SL}_{2}\wedge \mathrm{SL}_{2})\ast (\mathbb{G}_{m}\times \mathbb{G}_{m})/ C'(\mathbb{G}_{m}\times \mathbb{G}_{m})} \\
	{(\mathrm{SL}_{2}\ast \mathbb{G}_{m}) \wedge (\mathrm{SL}_{2}\ast \mathbb{G}_{m})} & {S^2\wedge\mathrm{SL}_{2}\wedge \mathrm{SL}_{2}\wedge (\mathbb{G}_{m}\times \mathbb{G}_{m})}
	\arrow["{\hat g}"', from=2-2, to=1-2]
	\arrow["\sim", from=2-2, to=3-2]
	\arrow[from=3-2, to=4-2]
	\arrow["{g'}", curve={height=-18pt}, from=4-1, to=1-2]
	\arrow["\delta"{description}, from=4-1, to=3-2]
	\arrow["\gamma", from=4-1, to=4-2]
\end{tikzcd}\] 
By direct inspection we see that \[\begin{tikzcd}
& {(S^1\wedge\mathrm{SL}_{2}\wedge \mathrm{SL}_{2})\ast (\mathbb{G}_{m}\times \mathbb{G}_{m})/ C'(\mathbb{G}_{m}\times \mathbb{G}_{m})} \\
{(\mathrm{SL}_{2}\ast \mathbb{G}_{m}) \wedge (\mathrm{SL}_{2}\ast \mathbb{G}_{m})} & {S^2\wedge\mathrm{SL}_{2}\wedge \mathrm{SL}_{2}\wedge (\mathbb{G}_{m}\times \mathbb{G}_{m})}
\arrow[from=1-2, to=2-2]
\arrow["\delta"{description}, from=2-1, to=1-2]
\arrow["\gamma", from=2-1, to=2-2]
\end{tikzcd}\] commutes. Next we look at \[\begin{tikzcd}
{(\mathrm{SL}_{2}\ast \mathbb{G}_{m}) \wedge (\mathrm{SL}_{2}\ast \mathbb{G}_{m})} & {S^2\wedge\mathrm{SL}_{2}\wedge \mathrm{SL}_{2}} \\
{(S^1\wedge\mathrm{SL}_{2}\wedge \mathrm{SL}_{2})\ast (\mathbb{G}_{m}\times \mathbb{G}_{m})/ C'(\mathbb{G}_{m}\times \mathbb{G}_{m})} & {(S^1\wedge\mathrm{SL}_{2}\wedge \mathrm{SL}_{2})\ast (\mathbb{G}_{m}\times \mathbb{G}_{m})}
\arrow["{g'}", from=1-1, to=1-2]
\arrow["\delta"', from=1-1, to=2-1]
\arrow["{\hat g}"', from=2-2, to=1-2]
\arrow["\sim", from=2-2, to=2-1]
\end{tikzcd}\]
We observe that $\hat g$ factors through $(S^1\wedge\mathrm{SL}_{2}\wedge \mathrm{SL}_{2})\ast (\mathbb{G}_{m}\times \mathbb{G}_{m})/ C'(\mathbb{G}_{m}\times \mathbb{G}_{m})$; hence there is a morphism $g'': (S^1\wedge\mathrm{SL}_{2}\wedge \mathrm{SL}_{2})\ast (\mathbb{G}_{m}\times \mathbb{G}_{m})/ C'(\mathbb{G}_{m}\times \mathbb{G}_{m})\rightarrow S^2\wedge\mathrm{SL}_{2}\wedge \mathrm{SL}_{2}$ such that \[\begin{tikzcd}
	& {S^2\wedge\mathrm{SL}_{2}\wedge \mathrm{SL}_{2}} \\
	{(S^1\wedge\mathrm{SL}_{2}\wedge \mathrm{SL}_{2})\ast (\mathbb{G}_{m}\times \mathbb{G}_{m})/ C'(\mathbb{G}_{m}\times \mathbb{G}_{m})} & {(S^1\wedge\mathrm{SL}_{2}\wedge \mathrm{SL}_{2})\ast (\mathbb{G}_{m}\times \mathbb{G}_{m})}
	\arrow["{g''}"{description}, from=2-1, to=1-2]
	\arrow["{\hat g}"', from=2-2, to=1-2]
	\arrow["\sim", from=2-2, to=2-1]
\end{tikzcd}\] commutes. Again by inspection we see that \[\begin{tikzcd}
{(\mathrm{SL}_{2}\ast \mathbb{G}_{m}) \wedge (\mathrm{SL}_{2}\ast \mathbb{G}_{m})} & {S^2\wedge\mathrm{SL}_{2}\wedge \mathrm{SL}_{2}} \\
{(S^1\wedge\mathrm{SL}_{2}\wedge \mathrm{SL}_{2})\ast (\mathbb{G}_{m}\times \mathbb{G}_{m})/ C'(\mathbb{G}_{m}\times \mathbb{G}_{m})}
\arrow["{g'}", from=1-1, to=1-2]
\arrow["\delta"', from=1-1, to=2-1]
\arrow["{g''}"{description}, from=2-1, to=1-2]
\end{tikzcd}\] commutes, too. It follows that the diagram \[\begin{tikzcd}
{(\mathrm{SL}_{2}\ast \mathbb{G}_{m}) \wedge (\mathrm{SL}_{2}\ast \mathbb{G}_{m})} & {S^2\wedge\mathrm{SL}_{2}\wedge \mathrm{SL}_{2}} \\
{(S^1\wedge\mathrm{SL}_{2}\wedge \mathrm{SL}_{2})\ast (\mathbb{G}_{m}\times \mathbb{G}_{m})/ C'(\mathbb{G}_{m}\times \mathbb{G}_{m})} & {(S^1\wedge\mathrm{SL}_{2}\wedge \mathrm{SL}_{2})\ast (\mathbb{G}_{m}\times \mathbb{G}_{m})}
\arrow["{g'}", from=1-1, to=1-2]
\arrow["\delta"', from=1-1, to=2-1]
\arrow["{\hat g}"', from=2-2, to=1-2]
\arrow["\sim", from=2-2, to=2-1]
\end{tikzcd}\] is commutative. Altogether we get a commutative diagram\[\begin{tikzcd}
& {S^2\wedge\mathrm{SL}_{2}\wedge \mathrm{SL}_{2}} \\
\\
& {(S^1\wedge\mathrm{SL}_{2}\wedge \mathrm{SL}_{2})\ast (\mathbb{G}_{m}\times \mathbb{G}_{m})} \\
{(\mathrm{SL}_{2}\ast \mathbb{G}_{m}) \wedge (\mathrm{SL}_{2}\ast \mathbb{G}_{m})} & {S^2\wedge\mathrm{SL}_{2}\wedge \mathrm{SL}_{2}\wedge (\mathbb{G}_{m}\times \mathbb{G}_{m})}
\arrow["{{\hat g}}"', from=3-2, to=1-2]
\arrow["{{g'}}"{description}, from=4-1, to=1-2]
\arrow["\gamma", from=4-1, to=4-2]
\arrow["{\tilde{p}^{-1}}"', from=4-2, to=3-2]
\end{tikzcd}\]
From this we can conclude that the morphism $\hat{g}\circ \tilde{p}^{-1}\circ \gamma \circ \gamma'^{-1} \circ \mathrm{id}_{S^2\wedge \mathrm{SL}_{2}\wedge \mathrm{SL}_{2}}\wedge \Delta$ is equal to $g' \circ \gamma'^{-1} \circ \mathrm{id}_{S^2\wedge \mathrm{SL}_{2}\wedge \mathrm{SL}_{2}}\wedge \Delta$. Via the identifications $(\ast)$ and $(\ast \ast)$ on page 20 the morphism $g' \circ \gamma'^{-1} \circ \mathrm{id}_{S^2\wedge \mathrm{SL}_{2}\wedge \mathrm{SL}_{2}}\wedge \Delta$ can be identified with the following composition \[\begin{tikzcd}
	{S^4\wedge\mathbb{G}_{m}^5} & {S^4\wedge\mathbb{G}_{m}^6} & {S^4\wedge\mathbb{G}_{m}^6} & {S^4\wedge\mathbb{G}_{m}^5} & {S^4\wedge\mathbb{G}_{m}^5} \\
	&&&& {S^4\wedge\mathbb{G}_{m}^4} \\
	&&&& {S^4\wedge\mathbb{G}_{m}^4}
	\arrow["{{\Delta_{4+(6)}}}", from=1-1, to=1-2]
	\arrow["{{\textcircled{1}}}", from=1-2, to=1-3]
	\arrow["{{\eta_{4+(5)}}}", from=1-3, to=1-4]
	\arrow["{{\textcircled{2}}}", from=1-4, to=1-5]
	\arrow["{{\eta_{4+(4)}}}", from=1-5, to=2-5]
	\arrow["{{\textcircled{3}}}", from=2-5, to=3-5]
\end{tikzcd}\]
The morphism $\textcircled{1}$ is given by $t_1\wedge t_2\wedge t_3\wedge t_4\wedge x_1\wedge x_2\wedge x_3\wedge x_4\wedge x_5\wedge x_6\mapsto t_2\wedge t_4\wedge t_1\wedge t_3\wedge x_1\wedge x_2\wedge x_5\wedge x_3\wedge x_4\wedge x_6 $. The permutation of $S^4$ given by $t_1\wedge t_2\wedge t_3\wedge t_4\mapsto t_2\wedge t_4\wedge t_1\wedge t_3$ is equal to $-\mathrm{id}_{S^4}$. Together with Corollary~\ref{Corollary 2.3} we see that $\textcircled{1}$ is just $-\mathrm{id}_{4+(6)}$. The morphism $\textcircled{2}$ is defined by $t_1\wedge t_2\wedge t_3\wedge t_4\wedge x_1\wedge x_2\wedge x_3\wedge x_4\wedge x_5\mapsto t_3\wedge t_4\wedge t_1\wedge t_2\wedge  x_3\wedge x_4\wedge x_5\wedge x_1\wedge x_2$. It is just $\mathrm{id}_{4+(5)}$. At the end the morphism $\textcircled{3}$ is given by $t_1\wedge t_2\wedge t_3\wedge t_4\wedge x_1\wedge x_2\wedge x_3\wedge x_4\mapsto t_1\wedge t_3\wedge t_2\wedge t_4\wedge x_3\wedge x_4\wedge x_1\wedge x_2$ which is -$\mathrm{id}_{4+(4)}$. It follows that the morphism $g' \circ \gamma'^{-1} \circ \mathrm{id}_{S^2\wedge \mathrm{SL}_{2}\wedge \mathrm{SL}_{2}}\wedge \Delta$ can be identified with $\eta_{4+(4)}\circ \eta_{4+(5)}\circ \Delta_{4+(6)}$.\\
\end{proof} 

Finally, we can prove Proposition~\ref{Proposition 3.2}. Via the commutative diagram on page 19 we show that it suffies to prove that the composite $\hat{g}\circ \tilde{p}^{-1}\circ \hat{j}^{-1}\circ\hat{j}\circ \mathrm{id}_{S^2\wedge \mathrm{SL}_{2}\wedge \mathrm{SL}_{2}}\wedge\tilde\Delta$ is nullhomotopic. This composite can be written as the sum of the following three morphisms \begin{align*}
	\hat{g}\circ \tilde{p}^{-1}\circ \hat{j}^{-1}\circ i_1\circ \pi_1\circ \mathrm{id}_{S^2\wedge \mathrm{SL}_{2}\wedge \mathrm{SL}_{2}}\wedge\tilde\Delta, \\
	\hat{g}\circ \tilde{p}^{-1}\circ \hat{j}^{-1}\circ i_2\circ \pi_2 \circ\mathrm{id}_{S^2\wedge \mathrm{SL}_{2}\wedge \mathrm{SL}_{2}}\wedge\tilde\Delta, \\
	\hat{g}\circ \tilde{p}^{-1}\circ \hat{j}^{-1}\circ i_3\circ \pi_3\circ \mathrm{id}_{S^2\wedge \mathrm{SL}_{2}\wedge \mathrm{SL}_{2}}\wedge\tilde\Delta \cdot
\end{align*}

Via the identifications  $(\ast)$ and $(\ast \ast)$ on page 20 and by Proposition~\ref{Proposition 3.3} and Proposition~\ref{Proposition 3.4} the sum of these three morphisms can be identified with $2\eta_{4+(4)}+ \eta_{4+(4)}\circ \eta_{4+(5)}\circ \Delta_{4+(6)}$.\\

\begin{cor} The morphism $2\eta_{4+(4)}+ \eta_{4+(4)}\circ \eta_{4+(5)}\circ \Delta_{4+(6)}$ is equal to 0 in $\mathcal{H}_{\bullet}(\mathbb{Z})$.\\
	\end{cor}
\begin{proof}
	By Lemma~\ref{Lemma 2.7} we have $1_{1+(2)}+\eta_{1+(2)}\circ \Delta_{1+(3)}=-\epsilon_{1+(2)}$. From this relation we obtain $1_{4+(5)}+\eta_{4+(5)}\circ \Delta_{4+(6)}=-\epsilon_{4+(5)}$. The morphism $2\eta_{4+(4)}+ \eta_{4+(4)}\circ \eta_{4+(5)}\circ \Delta_{4+(6)}$ is also equal to the composite $\eta_{4+(4)}\circ (2_{4+(5)}+\eta_{4+(5)}\circ \Delta_{4+(6)})=\eta_{4+(4)}\circ (1_{4+(5)}-\epsilon_{4+(5)})=\eta_{4+(4)}\circ h_{4+(5)} $. In the author's doctoral thesis \cite[Proposition 3.4.9]{dong2024motivictodabrackets} the relation $\eta_{1+(2)}\circ h_{1+(3)}=0$ is proved. From this relation we get $\eta_{4+(4)}\circ h_{4+(5)}=0$.\\
	\end{proof}

\begin{cor}\label{Corollary 3.6}
The composite $\eta_{3+(2)}\circ \nu_{3+(3)}$ is nullhomotopic.\\	
	\end{cor}

\begin{proof}
By Proposition~\ref{Proposition 3.1} and Proposition~\ref{Proposition 3.2} we see that the composite \[\begin{tikzcd}
	{S^1\wedge(\mathrm{SL}_{2}\ast \mathrm{SL}_{2})\wedge\mathbb{G}_{m}} \\
	\\
	{S^2\wedge\mathrm{SL}_{2}\wedge \mathbb{G}_{m}} \\
	{S^1\wedge(\mathrm{SL}_{2}\ast \mathbb{G}_{m})} & {S^2\wedge\mathrm{SL}_{2}}
	\arrow["{{\mathrm{id}_{S^1}\wedge \tilde\xi \wedge \mathrm{id}_{\mathbb{G}_{m}}}}"', from=1-1, to=3-1]
	\arrow["\sim"', from=3-1, to=4-1]
	\arrow["{{\mathrm{id}_{S^1}\wedge \tilde\alpha }}"', from=4-1, to=4-2]
\end{tikzcd}\] is nullhomotopic. Via the isomorphism $\lambda$ on page 14 and the isomorphisms $(\ast)$ and $(\ast \ast)$ on page 20 this composite can be identified with  $\eta_{3+(2)}\circ \nu_{3+(3)}$.\\
\end{proof}

\begin{prop}\label{Proposition 3.7} We have $\eta_{3+(3)}\circ \nu_{3+(4)}= \nu_{3+(3)}\circ \eta_{4+(5)}$. It follows that $\nu_{3+(3)}\circ \eta_{4+(5)}$ is nullhomotopic.\\
\end{prop}

\begin{proof}
We first look at $\eta_{4+(5)}$. It is a morphism from $S^4\wedge \mathbb{G}_{m}^6$ to $S^4\wedge \mathbb{G}_{m}^5$. Let $\phi:S^4\wedge \mathbb{G}_{m}^6\rightarrow S^4\wedge \mathbb{G}_{m}^6 $ be the isomorphism given by $t_1\wedge t_2\wedge t_3\wedge t_4\wedge x_1\wedge x_2\wedge x_3\wedge x_4\wedge x_5\wedge x_6\mapsto t_1\wedge t_2\wedge t_3\wedge t_4\wedge x_5\wedge x_6\wedge x_3\wedge x_4\wedge x_1\wedge x_2$. By Corollary~\ref{Corollary 2.3} the isomorphism $\phi$ is equal to $\mathrm{id}_{S^4\wedge \mathbb{G}_{m}^6}$. Then let $\psi: S^4\wedge \mathbb{G}_{m}^6\rightarrow S^4\wedge \mathbb{G}_{m}^6 $ be the isomorphism defined by $t_1\wedge t_2\wedge t_3\wedge t_4\wedge x_1\wedge x_2\wedge x_3\wedge x_4\wedge x_5\wedge x_6\mapsto t_2\wedge t_3\wedge t_4\wedge t_1\wedge x_1\wedge x_2\wedge x_3\wedge x_4\wedge x_5\wedge x_6$. In particular, $\psi$ is $-\mathrm{id}_{S^4\wedge \mathbb{G}_{m}^6}$. Similarly we define $\psi': S^4\wedge \mathbb{G}_{m}^5\rightarrow S^4\wedge \mathbb{G}_{m}^5 $ by $t_1\wedge t_2\wedge t_3\wedge t_4\wedge x_1\wedge x_2\wedge x_3\wedge x_4\wedge x_5\mapsto t_4\wedge t_1\wedge t_2\wedge t_3\wedge x_1\wedge x_2\wedge x_3\wedge x_4\wedge x_5$. It is equal to $-\mathrm{id}_{S^4\wedge \mathbb{G}_{m}^5}$. Moreover there is also a morphism $\phi':S^4\wedge \mathbb{G}_{m}^5\rightarrow S^4\wedge \mathbb{G}_{m}^5$ given by $t_1\wedge t_2\wedge t_3\wedge t_4\wedge x_1\wedge x_2\wedge x_3\wedge x_4\wedge x_5\mapsto t_1\wedge t_2\wedge t_3\wedge t_4\wedge x_2\wedge x_3\wedge x_4\wedge x_5\wedge x_1$. It is equal to $\mathrm{id}_{S^4\wedge \mathbb{G}_{m}^5}$. Therefore we get $\eta_{4+(5)}= \phi'\circ \psi'\circ \eta_{4+(5)}\circ \psi \circ \phi $.\\

The composite $\phi'\circ \psi'\circ \eta_{4+(5)}\circ \psi \circ \phi$ can also be written as the following  composite: \[\begin{tikzcd}
	{S^4\wedge \mathbb{G}_{m}^6} \\
	{S^3\wedge \mathbb{G}_{m}^4\wedge S^1\wedge\mathbb{G}_{m}^2} \\
	{S^3\wedge \mathbb{G}_{m}^4\wedge S^1\wedge\mathbb{G}_{m}} \\
	{S^4\wedge \mathbb{G}_{m}^5}
	\arrow["{\textcircled{1}}", from=1-1, to=2-1]
	\arrow["{\mathrm{id}_{S^3\wedge \mathbb{G}_{m}^4}\wedge\eta_{1+(2)}}", from=2-1, to=3-1]
	\arrow["{\textcircled{2}}", from=3-1, to=4-1]
\end{tikzcd}\]
where $\textcircled{1}$ is given by $t_1\wedge t_2\wedge t_3\wedge t_4\wedge x_1\wedge x_2\wedge x_3\wedge x_4\wedge x_5\wedge x_6\mapsto t_2\wedge t_3\wedge t_4\wedge x_1\wedge x_2\wedge x_3\wedge x_4\wedge t_1\wedge x_5\wedge x_6$ and $\textcircled{2}$ is defined by $t_1\wedge t_2\wedge t_3\wedge x_1\wedge x_2\wedge x_3\wedge x_4\wedge t_4\wedge x_5\mapsto t_4\wedge t_1\wedge t_2\wedge t_4\wedge x_1\wedge x_2\wedge x_3\wedge x_4\wedge x_5$.\\

Next we look at the composite \[\begin{tikzcd}
	{S^4\wedge \mathbb{G}_{m}^6} \\
	{S^3\wedge \mathbb{G}_{m}^4\wedge S^1\wedge\mathbb{G}_{m}^2} \\
	{S^3\wedge \mathbb{G}_{m}^4\wedge S^1\wedge\mathbb{G}_{m}} \\
	{S^4\wedge \mathbb{G}_{m}^5} \\
	{S^3\wedge \mathbb{G}_{m}^3} & \cdot
	\arrow["{\textcircled{1}}", from=1-1, to=2-1]
	\arrow["{\mathrm{id}_{S^3\wedge \mathbb{G}_{m}^4}\wedge\eta_{1+(2)}}", from=2-1, to=3-1]
	\arrow["{\textcircled{2}}", from=3-1, to=4-1]
	\arrow["{\nu_{3+(3)}}", from=4-1, to=5-1]
\end{tikzcd}\]
Recall that $\nu_{3+(3)}$ is the composite as follows: \[\begin{tikzcd}
	{S^4\wedge \mathbb{G}_{m}^5} \\
	{S^2\wedge S^1\wedge \mathbb{G}_{m}^2\wedge S^1\wedge \mathbb{G}_{m}^2\wedge \mathbb{G}_{m}} \\
	{S^2\wedge \mathrm{SL}_{2}\wedge \mathrm{SL}_{2}\wedge \mathbb{G}_{m}} \\
	{S^2\wedge \mathrm{SL}_{2}\wedge \mathbb{G}_{m}} \\
	{S^3\wedge \mathbb{G}_{m}^3}
	\arrow["{\mathrm{id}_{S^1}\wedge \lambda\wedge\mathrm{id}_{\mathbb{G}_{m}}}", from=1-1, to=2-1]
	\arrow["\sim", from=2-1, to=3-1]
	\arrow["{\mathrm{id}_{S^1}\wedge H(\xi)\wedge\mathrm{id}_{\mathbb{G}_{m}}}", from=3-1, to=4-1]
	\arrow["\sim", from=4-1, to=5-1]
\end{tikzcd}\] where $\lambda$ is defined on page 14. It follows from the definition of $\nu_{3+(3)}$ that we can exchange the order of $\nu$ and $\eta$  in \[\begin{tikzcd}
{S^4\wedge \mathbb{G}_{m}^6} \\
{S^3\wedge \mathbb{G}_{m}^4\wedge S^1\wedge\mathbb{G}_{m}^2} \\
{S^3\wedge \mathbb{G}_{m}^4\wedge S^1\wedge\mathbb{G}_{m}} \\
{S^4\wedge \mathbb{G}_{m}^5} \\
{S^3\wedge \mathbb{G}_{m}^3} & \cdot
\arrow["{\textcircled{1}}", from=1-1, to=2-1]
\arrow["{\mathrm{id}_{S^3\wedge \mathbb{G}_{m}^4}\wedge\eta_{1+(2)}}", from=2-1, to=3-1]
\arrow["{\textcircled{2}}", from=3-1, to=4-1]
\arrow["{\nu_{3+(3)}}", from=4-1, to=5-1]
\end{tikzcd}\] Thus we obtain $\eta_{3+(3)}\circ \nu_{3+(4)}= \nu_{3+(3)}\circ \eta_{4+(5)}$.\\
\end{proof}

\begin{cor}
The relation $\nu_{3+(3)}=-\epsilon_{3+(3)}\circ \nu_{3+(3)}$ holds.\\
\end{cor}

\begin{proof}
By Lemma 2.7 we have $1_{1+(2)}+\eta_{1+(2)}\circ \Delta_{1+(3)}=-\epsilon_{1+(2)}$. From this we obtain $1_{3+(3)}+\eta_{3+(3)}\circ \Delta_{3+(4)}=-\epsilon_{3+(3)}$. The suspended diagonal morphism $\Delta_{3+(4)}$ is a morphism from $S^3\wedge \mathbb{G}_{m}^3$ to $S^3\wedge \mathbb{G}_{m}^4$. Let $\omega: S^3\wedge \mathbb{G}_{m}^3\rightarrow S^3\wedge \mathbb{G}_{m}^3$ be the isomorphism given by $t_1\wedge t_2\wedge t_3\wedge x_1\wedge x_2\wedge x_3\mapsto  t_1\wedge t_2\wedge t_3\wedge x_3\wedge x_1\wedge x_2$. Again by Corollary~\ref{Corollary 2.3} it is equal to $\mathrm{id}_{S^3\wedge \mathbb{G}_{m}^3}$. Let $\omega': S^3\wedge \mathbb{G}_{m}^4\rightarrow S^3\wedge \mathbb{G}_{m}^4$ be the isomorphism given by $t_1\wedge t_2\wedge t_3\wedge x_1\wedge x_2\wedge x_3\wedge x_4\mapsto  t_1\wedge t_2\wedge t_3\wedge x_3\wedge x_4\wedge x_1\wedge x_2$. It follows also from Corollary~\ref{Corollary 2.3} that $\omega'$ equals to $\mathrm{id}_{S^3\wedge \mathbb{G}_{m}^4}$.\\

Now we see that $\Delta_{3+(4)}$ is equal to $\omega'\circ \Delta_{3+(4)} \circ \omega$. The composite $\omega'\circ \Delta_{3+(4)} \circ \omega$ is just $\mathrm{id}_{S^3\wedge \mathbb{G}_{m}^2}\wedge \Delta_{(2)}$. Therefore we get $\Delta_{3+(4)}\circ \nu_{3+(3)}= \nu_{3+(4)}\circ (\mathrm{id}_{S^4\wedge \mathbb{G}_{m}^4}\wedge \Delta_{(2)})$. By Corollary~\ref{Corollary 3.6} the relation $\eta_{3+(3)}\circ \Delta_{3+(4)}=0$ holds. Thus we obtain from $1_{3+(3)}+\eta_{3+(3)}\circ \Delta_{3+(4)}=-\epsilon_{3+(3)}$ the relation $\nu_{3+(3)}=-\epsilon_{3+(3)}\circ \nu_{3+(3)}$.\\
\end{proof}

At the end we prove the relation $\nu_{3+(2)}\circ \nu_{4+(4)}=-\nu_{3+(2)}\circ \nu_{4+(4)}$.\\

\begin{prop}\label{Proposition 3.9}
	The relation $\nu_{3+(2)}\circ \nu_{4+(4)}=-\nu_{3+(2)}\circ \nu_{4+(4)}$ holds.\\
\end{prop}

\begin{proof}
Let $\Omega: S^4\wedge \mathbb{G}_{m}^4\rightarrow S^4\wedge \mathbb{G}_{m}^4$ be the isomorphism given by $t_1\wedge t_2\wedge t_3\wedge t_4\wedge  x_1\wedge x_2\wedge x_3\wedge x_4\wedge \mapsto  t_1\wedge t_3\wedge t_4\wedge t_2\wedge x_1\wedge x_2\wedge x_3\wedge x_4$. In $\mathcal{H}_{\bullet}(\mathbb{Z})$ it is equal to $\mathrm{id}_{S^4\wedge \mathbb{G}_{m}^4}$. Thus we get $\nu_{3+(2)}\circ \nu_{4+(4)}=\nu_{3+(2)}\circ\Omega\circ  \nu_{4+(4)}$. The composite $\nu_{3+(2)}\circ\Omega\circ  \nu_{4+(4)}$ can be written as
\[\begin{tikzcd}
	{S^5\wedge \mathbb{G}_{m}^6} \\
	{S^2\wedge S^1\wedge \mathbb{G}_{m}^2\wedge S^1\wedge \mathbb{G}_{m}^2\wedge S^1\wedge \mathbb{G}_{m}^2} \\
	{S^1\wedge S^1\wedge\mathrm{SL}_2\wedge \mathrm{SL}_2\wedge \mathrm{SL}_2} \\
	{S^1\wedge S^1\wedge\mathrm{SL}_2\wedge \mathrm{SL}_2} \\
	{S^1\wedge S^1\wedge\mathrm{SL}_2} \\
	{S^3\wedge \mathbb{G}_{m}^2}
	\arrow["{{\textcircled{1}}}", from=1-1, to=2-1]
	\arrow["\sim", from=2-1, to=3-1]
	\arrow["{{\mathrm{id}_{S^1}\wedge H(\xi)\wedge \mathrm{id}_{\mathrm{SL}_2} }}", from=3-1, to=4-1]
	\arrow["{{\mathrm{id}_{S^1}\wedge H(\xi) }}", from=4-1, to=5-1]
	\arrow["\sim", from=5-1, to=6-1]
\end{tikzcd}\] where the isomorphism $\textcircled{1}$ is given by $t_1\wedge t_2\wedge t_3\wedge t_4\wedge t_5\wedge  x_1\wedge x_2\wedge x_3\wedge x_4\wedge x_5\wedge x_6 \mapsto  t_1\wedge t_3\wedge t_4\wedge  x_1\wedge x_2\wedge t_5\wedge x_3\wedge x_4\wedge t_2\wedge x_5\wedge x_6$. We first would like to show that $(\mathrm{id}_{S^1}\wedge H(\xi))\circ (\mathrm{id}_{S^1}\wedge H(\xi)\wedge \mathrm{id}_{\mathrm{SL}_2})$ equals to the composite \[\begin{tikzcd}
	{S^1\wedge S^1\wedge\mathrm{SL}_2\wedge \mathrm{SL}_2\wedge \mathrm{SL}_2} \\
	{S^1\wedge\mathrm{SL}_2\wedge S^1\wedge\mathrm{SL}_2\wedge \mathrm{SL}_2} \\
	{S^1\wedge\mathrm{SL}_2\wedge S^1\wedge\mathrm{SL}_2} \\
	{S^1\wedge S^1\wedge\mathrm{SL}_2\wedge\mathrm{SL}_2} \\
	{S^1\wedge S^1\wedge\mathrm{SL}_2}
	\arrow["{\textcircled{2}}", from=1-1, to=2-1]
	\arrow["{\mathrm{id}_{S^1\wedge\mathrm{SL}_2}\wedge H(\xi)}", from=2-1, to=3-1]
	\arrow["{\textcircled{3}}", from=3-1, to=4-1]
	\arrow["{\mathrm{id}_{S^1}\wedge H(\xi)}", from=4-1, to=5-1]
\end{tikzcd}\] where the morphism $\textcircled{2}$ is given by  $t_1\wedge t_2\wedge  x_1\wedge x_2\wedge x_3 \mapsto  t_2\wedge x_1\wedge t_1\wedge  x_3\wedge x_3$ and $\textcircled{3}$ is defined by $t_1\wedge x_1\wedge  t_2\wedge x_2 \mapsto  t_2\wedge t_1\wedge x_1\wedge  x_2$. We first observe that this composite can also be written as \[\begin{tikzcd}
	{S^1\wedge S^1\wedge\mathrm{SL}_2\wedge \mathrm{SL}_2\wedge \mathrm{SL}_2} \\
	{S^1\wedge S^1\wedge\mathrm{SL}_2\wedge \mathrm{SL}_2\wedge \mathrm{SL}_2} \\
	{S^1\wedge\mathrm{SL}_2\wedge S^1\wedge\mathrm{SL}_2\wedge \mathrm{SL}_2} \\
	{S^1\wedge\mathrm{SL}_2\wedge S^1\wedge\mathrm{SL}_2} \\
	{S^1\wedge\mathrm{SL}_2\wedge S^1\wedge\mathrm{SL}_2} \\
	{S^1\wedge S^1\wedge\mathrm{SL}_2\wedge\mathrm{SL}_2} \\
	{S^1\wedge S^1\wedge\mathrm{SL}_2} & \cdot
	\arrow["\tau", from=1-1, to=2-1]
	\arrow["{{\textcircled{2}}'}", from=2-1, to=3-1]
	\arrow["{{\mathrm{id}_{S^1\wedge\mathrm{SL}_2}\wedge H(\xi)}}", from=3-1, to=4-1]
	\arrow["{\tau'}", from=4-1, to=5-1]
	\arrow["{{\textcircled{3}'}}", from=5-1, to=6-1]
	\arrow["{{\mathrm{id}_{S^1}\wedge H(\xi)}}", from=6-1, to=7-1]
\end{tikzcd}\] The morphisms $\tau$ and $\tau'$ are induced by the permutation of the two copies of $S^1$. The morphism $\textcircled{2}'$ is given by  $t_1\wedge t_2\wedge  x_1\wedge x_2\wedge x_3 \mapsto  t_1\wedge x_1\wedge t_2\wedge  x_3\wedge x_3$ and $\textcircled{3}'$ is given by $t_1\wedge x_1\wedge  t_2\wedge x_2 \mapsto  t_1\wedge t_2\wedge x_1\wedge  x_2$. It follows that this composite equals to $(\mathrm{id}_{S^1}\wedge H(\xi))\circ \textcircled{3}'\circ (\mathrm{id}_{S^1\wedge\mathrm{SL}_2}\wedge H(\xi))\circ \textcircled{2}' $. Hence we only need to show that the two composites $(\mathrm{id}_{S^1}\wedge H(\xi))\circ \textcircled{3}'\circ (\mathrm{id}_{S^1\wedge\mathrm{SL}_2}\wedge H(\xi))\circ \textcircled{2}' $ and $(\mathrm{id}_{S^1}\wedge H(\xi))\circ (\mathrm{id}_{S^1}\wedge H(\xi)\wedge \mathrm{id}_{\mathrm{SL}_2})$ coincide.\\

For this we reveal the morphisms involved in $(\mathrm{id}_{S^1}\wedge H(\xi))\circ (\mathrm{id}_{S^1}\wedge H(\xi)\wedge \mathrm{id}_{\mathrm{SL}_2})$ as follows \[\begin{tikzcd}
	{S^1\wedge S^1\wedge\mathrm{SL}_2\wedge \mathrm{SL}_2\wedge \mathrm{SL}_2} \\
	{S^1\wedge (\mathrm{SL}_2\ast\mathrm{SL}_2)\wedge \mathrm{SL}_2} \\
	{S^1\wedge S^1\wedge\mathrm{SL}_2\wedge \mathrm{SL}_2} \\
	{S^1\wedge (\mathrm{SL}_2\ast\mathrm{SL}_2)} \\
	{S^1\wedge S^1\wedge\mathrm{SL}_2} & \cdot
	\arrow["\sim", from=1-1, to=2-1]
	\arrow["{\mathrm{id}_{S^1}\wedge \tilde \xi\wedge \mathrm{id}_{\mathrm{SL}_2}}", from=2-1, to=3-1]
	\arrow["\sim", from=3-1, to=4-1]
	\arrow["{\mathrm{id}_{S^1}\wedge \tilde \xi}", from=4-1, to=5-1]
\end{tikzcd}\] Recall that $\tilde \xi$ is just the morphism induced by the multiplication map $\xi$ on $\mathrm{SL}_2$. Then we consider the diagram \[\begin{tikzcd}
	{S^1\wedge S^1\wedge\mathrm{SL}_2\wedge \mathrm{SL}_2\wedge \mathrm{SL}_2} && {\mathrm{SL}_2\ast(\mathrm{SL}_2\ast \mathrm{SL}_2)} \\
	{S^1\wedge (\mathrm{SL}_2\ast\mathrm{SL}_2)\wedge \mathrm{SL}_2} \\
	{S^1\wedge S^1\wedge\mathrm{SL}_2\wedge \mathrm{SL}_2} \\
	{S^1\wedge (\mathrm{SL}_2\ast\mathrm{SL}_2)} \\
	{S^1\wedge S^1\wedge\mathrm{SL}_2}
	\arrow["\sim", from=1-1, to=2-1]
	\arrow["\iota"', from=1-3, to=1-1]
	\arrow["{\iota'}", from=1-3, to=4-1]
	\arrow["{\mathrm{id}_{S^1}\wedge \tilde \xi\wedge \mathrm{id}_{\mathrm{SL}_2}}", from=2-1, to=3-1]
	\arrow["\sim", from=3-1, to=4-1]
	\arrow["{\mathrm{id}_{S^1}\wedge \tilde \xi}", from=4-1, to=5-1]
\end{tikzcd}\] where $\iota$ is given by $\overline{(s,A,(\overline{t,B,C}))}\mapsto s\wedge t\wedge A\wedge B\wedge C$ and $\iota'$ is induced by the pointed morphism $\mathrm{SL}_2\times (\mathrm{SL}_2\ast\mathrm{SL}_2); (A,\overline{(t,B,C)})\mapsto \overline{(t,\xi(A,B),C)}$. In particular, we want to prove that the diagram \[\begin{tikzcd}
{S^1\wedge S^1\wedge\mathrm{SL}_2\wedge \mathrm{SL}_2\wedge \mathrm{SL}_2} &&& {\mathrm{SL}_2\ast(\mathrm{SL}_2\ast \mathrm{SL}_2)} \\
{S^1\wedge (\mathrm{SL}_2\ast\mathrm{SL}_2)\wedge \mathrm{SL}_2} \\
{S^1\wedge S^1\wedge\mathrm{SL}_2\wedge \mathrm{SL}_2} && {S^1\wedge (\mathrm{SL}_2\ast\mathrm{SL}_2)}
\arrow["\iota"', from=1-4, to=1-1]
\arrow["{\iota'}", from=1-4, to=3-3]
\arrow["\sim", from=2-1, to=1-1]
\arrow["{{\mathrm{id}_{S^1}\wedge \tilde \xi\wedge \mathrm{id}_{\mathrm{SL}_2}}}", from=2-1, to=3-1]
\arrow["\sim"', from=3-3, to=3-1]
\end{tikzcd}\] commutes in $\mathcal{H}_{\bullet}(\mathbb{Z})$. In order to prove the commutativity we need here again the geometric realization functor developed in the author's doctoral thesis \cite[Section 2]{dong2024motivictodabrackets}. After applying the geometric realization functor $|\cdot|$ we obtain a diagram \[\begin{tikzcd}
{|S^1|\wedge |S^1|\wedge|\mathrm{SL}_2|\wedge |\mathrm{SL}_2|\wedge |\mathrm{SL}_2|} && {|\mathrm{SL}_2|\ast(|\mathrm{SL}_2|\ast |\mathrm{SL}_2|)} \\
{|S^1|\wedge (|\mathrm{SL}_2|\ast|\mathrm{SL}_2|)\wedge| \mathrm{SL}_2|} \\
{|S^1|\wedge |S^1|\wedge|\mathrm{SL}_2|\wedge |\mathrm{SL}_2|} && {|S^1|\wedge (|\mathrm{SL}_2|\ast|\mathrm{SL}_2|)}
\arrow["{{|\iota|}}"', from=1-3, to=1-1]
\arrow["{\sigma'}", from=1-3, to=2-1]
\arrow["{{|\iota'|}}", from=1-3, to=3-3]
\arrow["\sim", from=2-1, to=1-1]
\arrow["{{|{\mathrm{id}_{S^1}\wedge \tilde \xi\wedge \mathrm{id}_{\mathrm{SL}_2}}|}}", from=2-1, to=3-1]
\arrow["\sim"', from=3-3, to=3-1]
\end{tikzcd}\] where the weak equivalence $\sigma'$ is defined by
\begin{align*}
\overline{(c,A,  (\overline{b,B,C}))}\mapsto \begin{cases} cb\wedge \overline{(\frac{(1-b)c}{1-cb},A,B)}\wedge C & \text{if} \  c\neq 1 \ \text {or}  \ b\neq 1\ \\
\ast & \text{if} \ c= 1 \ \text {and}  \ b=1\
\end{cases} 
\end{align*} for all $A,B,C\in  \mathrm{SL}_2$. We notice that the pointed continuous map \begin{align*}
I/\partial I\wedge I/\partial I \rightarrow I/\partial I\wedge I/\partial I; c\wedge b\mapsto \begin{cases}
	cb\wedge\frac{(1-b)c}{1-cb}  & \text{if} \  c\neq 1 \ \text {or}  \ b\neq 1\ \\
	\ast & \text{if} \ c= 1 \ \text {and}  \ b=1\
\end{cases}
\end{align*}
is pointed homotopic to the identity. It follows that the \[\begin{tikzcd}
	{|S^1|\wedge |S^1|\wedge|\mathrm{SL}_2|\wedge |\mathrm{SL}_2|\wedge |\mathrm{SL}_2|} && {|\mathrm{SL}_2|\ast(|\mathrm{SL}_2|\ast |\mathrm{SL}_2|)} \\
	{|S^1|\wedge (|\mathrm{SL}_2|\ast|\mathrm{SL}_2|)\wedge| \mathrm{SL}_2|} \\
	{|S^1|\wedge |S^1|\wedge|\mathrm{SL}_2|\wedge |\mathrm{SL}_2|} && {|S^1|\wedge (|\mathrm{SL}_2|\ast|\mathrm{SL}_2|)}
	\arrow["{{|\iota|}}"', from=1-3, to=1-1]
	\arrow["{\sigma'}", from=1-3, to=2-1]
	\arrow["{{|\iota'|}}", from=1-3, to=3-3]
	\arrow["\sim", from=2-1, to=1-1]
	\arrow["{{|{\mathrm{id}_{S^1}\wedge \tilde \xi\wedge \mathrm{id}_{\mathrm{SL}_2}}|}}", from=2-1, to=3-1]
	\arrow["\sim"', from=3-3, to=3-1]
\end{tikzcd}\] commutes in the corresponding homotopy category. Thus the original diagram \[\begin{tikzcd}
{S^1\wedge S^1\wedge\mathrm{SL}_2\wedge \mathrm{SL}_2\wedge \mathrm{SL}_2} &&& {\mathrm{SL}_2\ast(\mathrm{SL}_2\ast \mathrm{SL}_2)} \\
{S^1\wedge (\mathrm{SL}_2\ast\mathrm{SL}_2)\wedge \mathrm{SL}_2} \\
{S^1\wedge S^1\wedge\mathrm{SL}_2\wedge \mathrm{SL}_2} && {S^1\wedge (\mathrm{SL}_2\ast\mathrm{SL}_2)}
\arrow["\iota"', from=1-4, to=1-1]
\arrow["{\iota'}", from=1-4, to=3-3]
\arrow["\sim", from=2-1, to=1-1]
\arrow["{{\mathrm{id}_{S^1}\wedge \tilde \xi\wedge \mathrm{id}_{\mathrm{SL}_2}}}", from=2-1, to=3-1]
\arrow["\sim"', from=3-3, to=3-1]
\end{tikzcd}\] is commutative in $\mathcal{H}_{\bullet}(\mathbb{Z})$. Finally we look at the composite $(\mathrm{id}_{S^1}\wedge H(\xi))\circ \textcircled{3}'\circ (\mathrm{id}_{S^1\wedge\mathrm{SL}_2}\wedge H(\xi))\circ \textcircled{2}'$. By revealing the morphisms involved in this composite we get the following  diagram \[\begin{tikzcd}
{S^1\wedge S^1\wedge\mathrm{SL}_2\wedge \mathrm{SL}_2\wedge \mathrm{SL}_2} && {(\mathrm{SL}_2\ast\mathrm{SL}_2)\ast \mathrm{SL}_2} \\
{S^1\wedge\mathrm{SL}_2\wedge S^1\wedge\mathrm{SL}_2\wedge \mathrm{SL}_2} \\
{S^1\wedge\mathrm{SL}_2\wedge (\mathrm{SL}_2\ast \mathrm{SL}_2)} \\
{S^1\wedge\mathrm{SL}_2\wedge S^1\wedge\mathrm{SL}_2} \\
{S^1\wedge S^1\wedge\mathrm{SL}_2\wedge\mathrm{SL}_2} \\
{S^1\wedge  (\mathrm{SL}_2\ast \mathrm{SL}_2)} \\
{S^1\wedge S^1\wedge\mathrm{SL}_2}
\arrow["{{{\textcircled{2}}'}}", from=1-1, to=2-1]
\arrow["\kappa"', from=1-3, to=1-1]
\arrow["{\kappa'}", curve={height=-24pt}, from=1-3, to=6-1]
\arrow["\sim", from=2-1, to=3-1]
\arrow["{\mathrm{id}_{S^1\wedge\mathrm{SL}_2}\wedge \tilde \xi}", from=3-1, to=4-1]
\arrow["{{{\textcircled{3}'}}}", from=4-1, to=5-1]
\arrow["\sim", from=5-1, to=6-1]
\arrow["{\mathrm{id}_{S^1}\wedge \tilde \xi}", from=6-1, to=7-1]
\end{tikzcd}\] $\kappa$ is given by $\overline{(s,(\overline{t,A,B}),C)}\mapsto s\wedge t\wedge A\wedge B\wedge C$ and $\kappa'$ is induced by the pointed morphism $(\mathrm{SL}_2\ast \mathrm{SL}_2)\times \mathrm{SL}_2; (\overline{(t,A,B)}, C)\mapsto \overline{(t,A,\xi(BC))}$. By exactly the same arguments involving the geometric realization functor as before we can show that the diagram \[\begin{tikzcd}
{S^1\wedge S^1\wedge\mathrm{SL}_2\wedge \mathrm{SL}_2\wedge \mathrm{SL}_2} && {(\mathrm{SL}_2\ast\mathrm{SL}_2)\ast \mathrm{SL}_2} \\
{S^1\wedge\mathrm{SL}_2\wedge S^1\wedge\mathrm{SL}_2\wedge \mathrm{SL}_2} \\
{S^1\wedge\mathrm{SL}_2\wedge (\mathrm{SL}_2\ast \mathrm{SL}_2)} \\
{S^1\wedge\mathrm{SL}_2\wedge S^1\wedge\mathrm{SL}_2} \\
{S^1\wedge S^1\wedge\mathrm{SL}_2\wedge\mathrm{SL}_2} \\
{S^1\wedge  (\mathrm{SL}_2\ast \mathrm{SL}_2)}
\arrow["{{{\textcircled{2}}'}}", from=1-1, to=2-1]
\arrow["\kappa"', from=1-3, to=1-1]
\arrow["{\kappa'}", curve={height=-24pt}, from=1-3, to=6-1]
\arrow["\sim", from=2-1, to=3-1]
\arrow["{\mathrm{id}_{S^1\wedge\mathrm{SL}_2}\wedge \tilde \xi}", from=3-1, to=4-1]
\arrow["{{{\textcircled{3}'}}}", from=4-1, to=5-1]
\arrow["\sim", from=5-1, to=6-1]
\end{tikzcd}\] commutes.\\

Next we use the geometric realization functor again and obtain a diagram \[\begin{tikzcd}
	{(|\mathrm{SL}_2|\ast|\mathrm{SL}_2|)\ast |\mathrm{SL}_2|} && {|\mathrm{SL}_2|\ast(|\mathrm{SL}_2|\ast |\mathrm{SL}_2|)} \\
	{|S^1|\wedge (|\mathrm{SL}_2|\ast |\mathrm{SL}_2|)} && {|S^1|\wedge (|\mathrm{SL}_2|\ast |\mathrm{SL}_2|)} \\
	{|S^1|\wedge |S^1|\wedge|\mathrm{SL}_2|}
	\arrow["{|\kappa'|}"', from=1-1, to=2-1]
	\arrow["{\sigma''}"', from=1-3, to=1-1]
	\arrow["{|\iota'|}", from=1-3, to=2-3]
	\arrow["{|\mathrm{id}_{S^1}\wedge \tilde \xi|}"', from=2-1, to=3-1]
	\arrow["{|\mathrm{id}_{S^1}\wedge \tilde \xi|}", from=2-3, to=3-1]
\end{tikzcd}\] where the weak equivalence $\sigma''$ is given by $\overline{(c,A,  (\overline{b,B,C}))}\mapsto \overline{(cb, \overline{(\frac{(1-b)c}{1-cb},A,B)},C})$ for all $c.b\in I/\partial I$ and $A,B,C\in  \mathrm{SL}_2$. Again we see that this diagram commutes by homotopy. Furthermore by the same arguments as before we also see that \[\begin{tikzcd}
	{(|\mathrm{SL}_2|\ast|\mathrm{SL}_2|)\ast |\mathrm{SL}_2|} && {|\mathrm{SL}_2|\ast(|\mathrm{SL}_2|\ast |\mathrm{SL}_2|)} \\
	& {|S^1|\wedge |S^1|\wedge|\mathrm{SL}_2|\wedge |\mathrm{SL}_2|\wedge |\mathrm{SL}_2|}
	\arrow["{|\kappa|}", from=1-1, to=2-2]
	\arrow["{\sigma''}"', from=1-3, to=1-1]
	\arrow["{|\iota|}"', from=1-3, to=2-2]
\end{tikzcd}\] commutes in the corresponding homotopy category. Altogether we get $|(\mathrm{id}_{S^1}\wedge H(\xi))\circ (\mathrm{id}_{S^1}\wedge H(\xi)\wedge \mathrm{id}_{\mathrm{SL}_2})|\circ |\iota|=|\iota'|\circ |\mathrm{id}_{S^1}\wedge \tilde \xi|=|\kappa'|\circ |\mathrm{id}_{S^1}\wedge \tilde \xi|\circ \sigma''= |(\mathrm{id}_{S^1}\wedge H(\xi))\circ \textcircled{3}'\circ (\mathrm{id}_{S^1\wedge\mathrm{SL}_2}\wedge H(\xi))\circ \textcircled{2}'|\circ|\kappa|\circ \sigma''=  |(\mathrm{id}_{S^1}\wedge H(\xi))\circ \textcircled{3}'\circ (\mathrm{id}_{S^1\wedge\mathrm{SL}_2}\wedge H(\xi))\circ \textcircled{2}'|\circ |\iota|  $. As $|\iota|$ is a weak equivalence, we obtain $|(\mathrm{id}_{S^1}\wedge H(\xi))\circ (\mathrm{id}_{S^1}\wedge H(\xi)\wedge \mathrm{id}_{\mathrm{SL}_2})|=|(\mathrm{id}_{S^1}\wedge H(\xi))\circ \textcircled{3}'\circ (\mathrm{id}_{S^1\wedge\mathrm{SL}_2}\wedge H(\xi))\circ \textcircled{2}'|$. It follows that the two composites $(\mathrm{id}_{S^1}\wedge H(\xi))\circ \textcircled{3}'\circ (\mathrm{id}_{S^1\wedge\mathrm{SL}_2}\wedge H(\xi))\circ \textcircled{2}' $ and $(\mathrm{id}_{S^1}\wedge H(\xi))\circ (\mathrm{id}_{S^1}\wedge H(\xi)\wedge \mathrm{id}_{\mathrm{SL}_2})$ coincide because the left derived functor $|\cdot|$ is an equivalence of the homotopy categories.\\

Recall that $(\mathrm{id}_{S^1}\wedge H(\xi))\circ \textcircled{3}'\circ (\mathrm{id}_{S^1\wedge\mathrm{SL}_2}\wedge H(\xi))\circ \textcircled{2}'$ 	is just the composite \[\begin{tikzcd}
	{S^1\wedge S^1\wedge\mathrm{SL}_2\wedge \mathrm{SL}_2\wedge \mathrm{SL}_2} \\
	{S^1\wedge\mathrm{SL}_2\wedge S^1\wedge\mathrm{SL}_2\wedge \mathrm{SL}_2} \\
	{S^1\wedge\mathrm{SL}_2\wedge S^1\wedge\mathrm{SL}_2} \\
	{S^1\wedge S^1\wedge\mathrm{SL}_2\wedge\mathrm{SL}_2} \\
	{S^1\wedge S^1\wedge\mathrm{SL}_2}
	\arrow["{\textcircled{2}}", from=1-1, to=2-1]
	\arrow["{\mathrm{id}_{S^1\wedge\mathrm{SL}_2}\wedge H(\xi)}", from=2-1, to=3-1]
	\arrow["{\textcircled{3}}", from=3-1, to=4-1]
	\arrow["{\mathrm{id}_{S^1}\wedge H(\xi)}", from=4-1, to=5-1]
\end{tikzcd}\] 
We precompose it with \[\begin{tikzcd}
{S^5\wedge \mathbb{G}_{m}^6} \\
{S^2\wedge S^1\wedge \mathbb{G}_{m}^2\wedge S^1\wedge \mathbb{G}_{m}^2\wedge S^1\wedge \mathbb{G}_{m}^2} \\
{S^1\wedge S^1\wedge\mathrm{SL}_2\wedge \mathrm{SL}_2\wedge \mathrm{SL}_2}
\arrow["{{{\textcircled{1}}}}", from=1-1, to=2-1]
\arrow["\sim", from=2-1, to=3-1]
\end{tikzcd}\] and would like to show that \[\begin{tikzcd}
{S^5\wedge \mathbb{G}_{m}^6} \\
{S^2\wedge S^1\wedge \mathbb{G}_{m}^2\wedge S^1\wedge \mathbb{G}_{m}^2\wedge S^1\wedge \mathbb{G}_{m}^2} \\
{S^1\wedge S^1\wedge\mathrm{SL}_2\wedge \mathrm{SL}_2\wedge \mathrm{SL}_2} \\
{S^1\wedge\mathrm{SL}_2\wedge S^1\wedge\mathrm{SL}_2\wedge \mathrm{SL}_2} \\
{S^1\wedge\mathrm{SL}_2\wedge S^1\wedge\mathrm{SL}_2} \\
{S^1\wedge S^1\wedge\mathrm{SL}_2\wedge\mathrm{SL}_2} \\
{S^1\wedge S^1\wedge\mathrm{SL}_2}
\arrow["{\textcircled{1}}", from=1-1, to=2-1]
\arrow["\sim", from=2-1, to=3-1]
\arrow["{{\textcircled{2}}}", from=3-1, to=4-1]
\arrow["{{\mathrm{id}_{S^1\wedge\mathrm{SL}_2}\wedge H(\xi)}}", from=4-1, to=5-1]
\arrow["{{\textcircled{3}}}", from=5-1, to=6-1]
\arrow["{{\mathrm{id}_{S^1}\wedge H(\xi)}}", from=6-1, to=7-1]
\end{tikzcd}\] is $-\nu_{3+(2)}\circ \nu_{4+(4)}$. By unraveling the definition of the morphisms in the composite above we see that this composite is equal to \[\begin{tikzcd}
{S^5\wedge \mathbb{G}_{m}^6} \\
{S^1\wedge S^1\wedge \mathbb{G}_{m}^2\wedge S^1\wedge S^1\wedge \mathbb{G}_{m}^2\wedge S^1\wedge \mathbb{G}_{m}^2} \\
{S^1\wedge S^1\wedge \mathbb{G}_{m}^2\wedge S^1\wedge S^1\wedge \mathbb{G}_{m}^2} \\
{S^1\wedge S^1\wedge S^1\wedge \mathbb{G}_{m}^2\wedge S^1\wedge \mathbb{G}_{m}^2} \\
{S^1\wedge S^1\wedge S^1\wedge \mathbb{G}_{m}^2}
\arrow["{\textcircled{1}'}", from=1-1, to=2-1]
\arrow["{\mathrm{id}_{S^1\wedge S^1\wedge \mathbb{G}_{m}^2}\wedge\nu_{2+(2)}}", from=2-1, to=3-1]
\arrow["{\tau''}", from=3-1, to=4-1]
\arrow["{{\mathrm{id}_{S^1}\wedge \nu_{2+(2)}}}", from=4-1, to=5-1]
\end{tikzcd}\] where $\textcircled{1}'$ is given by $t_1\wedge t_2\wedge t_3\wedge t_4\wedge t_5\wedge  x_1\wedge x_2\wedge x_3\wedge x_4\wedge x_5\wedge x_6 \mapsto  t_3\wedge t_4\wedge  x_1\wedge x_2\wedge t_1\wedge t_5\wedge x_3\wedge x_4\wedge t_2\wedge x_5\wedge x_6$ and $\tau''$ is defined by $t_1\wedge t_2\wedge   x_1\wedge x_2\wedge t_3\wedge t_4\wedge x_3\wedge x_4\mapsto t_3\wedge t_1\wedge t_2\wedge   x_1\wedge x_2\wedge t_4\wedge x_3\wedge x_4 $. In particular this composite can also be written as \[\begin{tikzcd}
{S^5\wedge \mathbb{G}_{m}^6} \\
{S^5\wedge \mathbb{G}_{m}^6} \\
{S^1\wedge S^1\wedge S^1\wedge \mathbb{G}_{m}^2\wedge S^1\wedge \mathbb{G}_{m}^2\wedge S^1\wedge \mathbb{G}_{m}^2} \\
{S^1\wedge S^1\wedge S^1\wedge \mathbb{G}_{m}^2\wedge  S^1\wedge \mathbb{G}_{m}^2} \\
{S^1\wedge S^1\wedge S^1\wedge \mathbb{G}_{m}^2\wedge  S^1\wedge \mathbb{G}_{m}^2} \\
{S^1\wedge S^1\wedge S^1\wedge \mathbb{G}_{m}^2} & \cdot
\arrow["\mu", from=1-1, to=2-1]
\arrow["{\textcircled{1}}", from=2-1, to=3-1]
\arrow["{\mathrm{id}_{S^1}\wedge \nu_{2+(2)}\wedge\mathrm{id}_{S^1\wedge \mathbb{G}_{m}^2}}", from=3-1, to=4-1]
\arrow["{\mu'}", from=4-1, to=5-1]
\arrow["{{\mathrm{id}_{S^1}\wedge \nu_{2+(2)}}}", from=5-1, to=6-1]
\end{tikzcd}\]
The morphism $\mu$ is given by $t_1\wedge t_2\wedge t_3\wedge t_4\wedge t_5\wedge  x_1\wedge x_2\wedge x_3\wedge x_4\wedge x_5\wedge x_6 \mapsto  t_3\wedge t_4\wedge t_1\wedge t_5\wedge t_2\wedge x_3\wedge x_4\wedge x_5\wedge x_6\wedge x_1\wedge x_2$. By Corollary~\ref{Corollary 2.3} it is equal to $-\mathrm{id}_{S^5\wedge \mathbb{G}_{m}^6}$. The morphism $\mu'$ is given by $t_1\wedge t_2\wedge t_3\wedge x_1\wedge x_2\wedge t_4\wedge x_3\wedge x_4 \mapsto t_2\wedge t_1\wedge t_4\wedge x_3\wedge x_4\wedge t_3\wedge x_1\wedge x_2$. It is equal to the identity. Especially, the composite \[\begin{tikzcd}
	{S^5\wedge \mathbb{G}_{m}^6} \\
	{S^1\wedge S^1\wedge S^1\wedge \mathbb{G}_{m}^2\wedge S^1\wedge \mathbb{G}_{m}^2\wedge S^1\wedge \mathbb{G}_{m}^2} \\
	{S^1\wedge S^1\wedge S^1\wedge \mathbb{G}_{m}^2\wedge  S^1\wedge \mathbb{G}_{m}^2} \\
	{S^1\wedge S^1\wedge S^1\wedge \mathbb{G}_{m}^2\wedge  S^1\wedge \mathbb{G}_{m}^2} \\
	{S^1\wedge S^1\wedge S^1\wedge \mathbb{G}_{m}^2}
	\arrow["{{\textcircled{1}}}", from=1-1, to=2-1]
	\arrow["{{\mathrm{id}_{S^1}\wedge \nu_{2+(2)}\wedge\mathrm{id}_{S^1\wedge \mathbb{G}_{m}^2}}}", from=2-1, to=3-1]
	\arrow["{{\mu'}}", from=3-1, to=4-1]
	\arrow["{{{\mathrm{id}_{S^1}\wedge \nu_{2+(2)}}}}", from=4-1, to=5-1]
\end{tikzcd}\] equals to $\nu_{3+(2)}\circ \nu_{4+(4)}$ as $\mu'$
is the identity. Altogether we conclude that $\nu_{3+(2)}\circ \nu_{4+(4)}=-\nu_{3+(2)}\circ \nu_{4+(4)}$.\\
\end{proof}

\begin{cor}
The Toda brackets  $\{\eta_{3+(2)}, \nu_{3+(3)}, \eta _{4+(5)} \},  \\ \{\nu_{3+(3)}, \eta _{4+(5)}, \nu _{4+(6)} \}$ and $\{\eta_{3+(2)}, (1-\epsilon)_{3+(3)}, \nu_{3+(3)}\circ \nu_{4+(5)}\}$ are well-defined. Moreover, they are non-trivial as their complex realizations are non-trivial.\\
	\end{cor}

\begin{proof}
Let $\eta_{\mathrm{top}}$ denote the first topological Hopf map and $\nu_{\mathrm{top}}$ the second topological Hopf map. The complex realization functor (see \cite[Section A.4]{Panin2009}) sends the motivic Toda bracket  $\{\eta_{3+(2)}, \nu_{3+(3)}, \eta _{4+(5)} \}$ to the topological Toda bracket $\{\Sigma^3\eta_{\mathrm{top}}, \Sigma^2\nu_{\mathrm{top}}, \Sigma^7\eta _{\mathrm{top}} \}$. By \cite[Lemma 5.12]{Toda+1963} and the remark next to this lemma the topological Toda bracket \begin{align*}
\{\Sigma^3\eta_{\mathrm{top}}, \Sigma^2\nu_{\mathrm{top}}, \Sigma^7\eta _{\mathrm{top}} \}
\end{align*} is not trivial. Similarly, the complex realization functor sends the motivic Toda bracket $\{\nu_{3+(3)}, \eta _{4+(5)}, \nu _{4+(6)} \}$ to the topological Toda bracket $\{\Sigma^2\nu_{\mathrm{top}}, \Sigma^7\eta _{\mathrm{top}}, \Sigma^6\nu _{\mathrm{top}} \}$. By \cite[Lemma 6.2, Theorem 7.1]{Toda+1963} this topological Toda bracket is again not trivial. Finally the complex realization of \begin{align*}
\{\eta_{3+(2)}, (1-\epsilon)_{3+(3)}, \nu_{3+(3)}\circ \nu_{4+(5)}\}
\end{align*} is the topological Toda bracket $\{\Sigma^3\eta_{\mathrm{top}}, 2\mathrm{id}_{S^6}, \Sigma^2\nu_{\mathrm{top}}\circ \Sigma^5\nu_{\mathrm{top}}\}$. By \cite[(6.1)]{Toda+1963} this topological Toda bracket is not trivial.
\end{proof}

			\bibliographystyle{alpha} 
	\addcontentsline{toc}{section}{References}
	\bibliography{references} 
	\end{document}